\documentclass[oneside,english]{amsart}
\usepackage[T1]{fontenc}
\usepackage[latin9]{inputenc}
\usepackage{geometry}
\usepackage{textcomp}
\usepackage{mathrsfs}
\usepackage{amsbsy}
\usepackage{amstext}
\usepackage{amsthm}
\usepackage{amssymb}

\makeatletter
\numberwithin{equation}{section}
\numberwithin{figure}{section}
\theoremstyle{plain}
\newtheorem{thm}{\protect\theoremname}[section]
\theoremstyle{plain}
\newtheorem{conjecture}[thm]{\protect\conjecturename}
\theoremstyle{definition}
\newtheorem{example}[thm]{\protect\examplename}
\theoremstyle{remark}
\newtheorem{rem}[thm]{\protect\remarkname}
\theoremstyle{plain}
\newtheorem{lem}[thm]{\protect\lemmaname}
\theoremstyle{plain}
\newtheorem{prop}[thm]{\protect\propositionname}

\def\makebbb#1{
    \expandafter\gdef\csname#1\endcsname{
        \ensuremath{\Bbb{#1}}}
}\makebbb{R}\makebbb{N}\makebbb{Z}\makebbb{C}\makebbb{H}\makebbb{E}\makebbb{H}\makebbb{P}\makebbb{B}\makebbb{Q}\makebbb{E}\makebbb{E}
\makebbb{H}
\usepackage{babel}

\usepackage{babel}

\makeatother

\usepackage{babel}

\makeatother

\usepackage{babel}
\providecommand{\conjecturename}{Conjecture}
\providecommand{\examplename}{Example}
\providecommand{\lemmaname}{Lemma}
\providecommand{\propositionname}{Proposition}
\providecommand{\remarkname}{Remark}
\providecommand{\theoremname}{Theorem}

\begin{document}
\title{Kähler-Einstein metrics and Archimedean zeta functions }
\author{Robert J. Berman}
\begin{abstract}
While the existence of a unique Kähler-Einstein metric on a canonically
polarized manifold $X$ was established by Aubin and Yau already in
the 70s there are only a few explicit formulas available. In previous
work a probabilistic construction of the Kähler-Einstein metric was
introduced - involving canonical random point processes on $X$ -
which yields canonical approximations of the Kähler-Einstein metric,
expressed as explicit period integrals over a large number of products
of $X.$ Here it is shown that the conjectural extension to the case
when $X$ is a Fano variety suggests a zero-free property of the Archimedean
zeta functions defined by the partition functions of the probabilistic
model. A weaker zero-free property is also shown to be relevant for
the Calabi-Yau equation. The convergence in the case of log Fano curves
is settled, exploiting relations to the complex Selberg integral in
the orbifold case. Some intriguing relations to the zero-free property
of the local automorphic L-functions appearing in the Langlands program
and arithmetic geometry are also pointed out. These relations also
suggest a natural $p$-adic extension of the probabilistic approach.
\end{abstract}

\address{Robert J. Berman, Mathematical Sciences, Chalmers University of Technology
and the University of Gothenburg, SE-412 96 Göteborg, Sweden}
\email{robertb@chalmers.se}
\maketitle

\section{Introduction}

A metric $\omega$ on a compact complex manifold $X$ is said to be
\emph{Kähler-Einstein} if it has constant Ricci curvature:
\[
\text{Ric \ensuremath{\omega}=\ensuremath{-\beta}\ensuremath{\omega}}
\]
for some constant $\beta$ and $\omega$ is Kähler (i.e. parallel
translation preserves the complex structure on $X).$ Such metrics
play a prominent in current complex differential geometry and the
study of complex algebraic varieties, in particular in the context
of the Yau-Tian-Donaldson conjecture \cite{do3} and the Minimal Model
Program in birational algebraic geometry \cite{ko}. In \cite{berm8,berm8 comma 5}
a probabilistic construction of Kähler-Einstein metrics with negative
Ricci curvature on a complex projective algebraic variety $X$ was
introduced, where the Kähler-Einstein metric emerges from a canonical
random point process on $X.$ The random point process is defined
in terms of purely algebro-geometric data. Accordingly, one virtue
of this approach is that it generates new links between differential
geometry on the one hand and algebraic-geometry on the other. In the
present work it is, in particular, shown that the conjectural extension
to Kähler-Einstein metrics with positive Ricci curvature suggests
a zero-free property of the Archimedean zeta functions defined by
the partition functions of the probabilistic model. The particular
case of Kähler-Einstein metrics with conical singularities on the
Riemann sphere is settled, which from  the algebro-geometric perspective
corresponds to the case of log Fano curves. 

We start by providing some background on Kähler-Einstein metrics and
recapitulating the probabilistic approach to Kähler-Einstein metrics;
the reader is referred to the survey \cite{berm11} for more  background
and \cite{ber12} for relations to the Yau-Tian-Donaldson conjecture.
See also \cite{b-c-p} for connections to quantum gravity in the context
of the AdS/CFT correspondence and \cite{berm11b,du} for connections
to polynomial approximation theory and pluripotential theory in $\C^{n}.$ 

\subsection{Kähler-Einstein metrics}

The existence of a Kähler-Einstein metric on $X$ implies that the\emph{
canonical line bundle $K_{X}$} of $X$ (i.e. the top exterior power
of the cotangent bundle of $X)$ has a definite sign: 
\begin{equation}
\text{sign}(K_{X})=\text{sign}(\beta)\label{eq:sing of K X}
\end{equation}
 We will be using the standard terminology of positivity in complex
geometry: a line bundle $L$ is said to be \emph{positive,} $L>0,$
if it is ample and \emph{negative,} $L<0,$ if its dual is positive.
In analytic terms, $L>0$ iff $L$ carries some Hermitian metric with
strictly positive curvature. The standard additive notation for tensor
products of line bundles will be adopted. Accordingly, the dual of
$L$ is expressed as $-L.$ We will focus on the cases when $\beta\neq0.$
Then $X$ is automatically a complex projective algebraic manifold
and after a rescaling of the metric we may as well assume that $\beta=\pm1.$
For example, in the case when $X$ is a hypersurface in $\P^{n+1},$
cut out by a homogeneous polynomial of degree $d,$ 
\[
K_{X}>0\iff d>n+2,\,\,\,\,\,\,\,-K_{X}>0\iff d<n+2.
\]
In the case when $K_{X}>0$ the existence of a Kähler-Einstein metric
was established in the late seventies \cite{au,y}. The opposite case
$-K_{X}>0$ is the subject of the Yau-Tian-Donaldson conjecture, which
was settled only recently (see the survey \cite{do3}). However, these
are abstract existence results and there are very few explicit formulas
for Kähler-Einstein metrics on complex algebraic varieties available.
For example, even in the simplest case when $K_{X}>0$ and $X$ is
complex curve, $n=1,$ finding an explicit formula for the Kähler-Einstein
metric is equivalent to finding an explicit uniformization map from
the curve $X$ to the quotient $\H/G$ of the upper half-plane by
a discrete subgroup $G\subset SL(2,\R).$ This has only been achieved
for very special curves (such as the Klein quartic and Fermat curves),
using techniques originating in the classical works by Weierstrass,
Riemann, Fuchs, Schwartz, Klein, Poincaré,... Thus one virtue of the
probabilistic approach is that it yields canonical approximations
of the Kähler-Einstein metric on $X,$ expressed as essentially explicit
period type integrals formulas (see formula \ref{eq:def of omega k KX pos intro}
below). These are reminiscent of the aforementioned few explicit formulas
for Kähler-Einstein metrics, involving hypergeometric integrals (see
\cite[Section 2.1]{berm11}).

\subsection{The probabilistic approach}

First recall that that, in the case when $\beta\neq0,$ a Kähler-Einstein
metric $\omega_{KE}$ on $X$ can be readily recovered from its (normalized)
volume form $dV_{KE}:$
\[
\omega_{KE}=\frac{1}{\beta}\frac{i}{2\pi}\partial\bar{\partial}\log dV_{KE},
\]
where we have identified the volume form $dV$ with its local density,
defined with respect to a choice of local holomorphic coordinates
$z.$ The strategy of the probabilistic approach is to construct the
normalized volume form $dV_{KE}$ by a canonical sampling procedure
on $X.$ In other words, after constructing a canonical symmetric
probability measure $\mu^{(N)}$ on $X^{N}$ the goal is to show that
the corresponding \emph{empirical measure }
\[
\delta_{N}:=\frac{1}{N}\sum_{i=1}^{N}\delta_{x_{i}},
\]
 viewed as a random discrete measure on $X,$ converges in probability
as $N\rightarrow\infty,$ to the volume form $dV_{KE}$ of the Kähler-Einstein
metric $\omega_{KE}.$

\subsubsection{The case $\beta>0$}

When $K_{X}>0$ the canonical probability measure $\mu^{(N)}$ on
$X^{N},$ introduced in \cite{berm8}, is defined for a specific subsequence
of integers $N_{k}$ tending to infinity, the \emph{plurigenera }of
$X:$ 
\[
N_{k}:=\dim H^{0}(X,kK_{X}),
\]
 where $H^{0}(X,kK_{X})$ denotes the complex vector space of all
global holomorphic sections $s^{(k)}$ of the $k$ th tensor power
of the canonical line bundle $K_{X}\rightarrow X$ (called pluricanonical
forms). The assumption that $K_{X}>0$ ensures that $N_{k}\rightarrow\infty$,
as $k\rightarrow\infty.$ In terms of local holomorphic coordinates
$z\in\C^{n}$ on $X,$ a section $s^{(k)}$ of $kK_{X}\rightarrow X$
may be represented by local holomorphic functions $s^{(k)}$ on $X,$
such that $|s^{(k)}|^{2/k}$ transforms as a density on $X,$ i.e.
defines a measure on $X.$ The canonical symmetric probability measure
$\mu^{(N_{k})}$ on $X^{N_{k}}$ is concretely defined by 
\begin{equation}
\mu^{(N_{k})}:=\frac{1}{\mathcal{Z}_{N_{k}}}\left|\det S^{(k)}\right|^{2/k},\,\,\,\,\mathcal{Z}_{N_{k}}:=\int_{X^{N_{k}}}\left|\det S^{(k)}\right|^{2/k}\label{eq:canon prob measure intro}
\end{equation}
 where $\det S^{(k)}$ is the holomorphic section of the canonical
line bundle $(kK_{X^{N_{k}}})$ over $X^{N_{k}}$, defined by the
Slater determinant
\begin{equation}
(\det S^{(k)})(x_{1},x_{2},...,x_{N_{k}}):=\det(s_{i}^{(k)}(x_{j})),\label{eq:slater determinant}
\end{equation}
 in terms of a given basis $s_{i}^{(k)}$ in $H^{0}(X,kK_{X}).$ Under
a change of bases the section $\det S^{(k)}$ only changes by a multiplicative
complex constant (the determinant of the change of bases matrix on
$H^{0}(X,kK_{X})$) and so does the normalizing constant $\mathcal{Z}_{N_{k}}.$
As a result $\mu^{(N_{k})}$ is indeed canonical, i.e. independent
of the choice of bases. Moreover, it is completely encoded by algebro-geometric
data in the following sense: realizing $X$ as projective algebraic
subvariety the section $\det S^{(k)}$ can be identified with a homogeneous
polynomial, determined by the coordinate ring of $X$ (or more precisely,
the degree $k$ component of the canonical ring of $X).$

The following convergence result was shown in \cite{berm8}:
\begin{thm}
\label{thm:ke intro}Let $X$ be a compact complex manifold with positive
canonical line bundle $K_{X}.$ Then the empirical measures $\delta_{N_{k}}$
of the corresponding canonical random point processes on $X$ converge
in probability, as $N_{k}\rightarrow\infty,$ towards the normalized
volume form $dV_{KE}$ of the unique Kähler-Einstein metric $\omega_{KE}$
on $X.$ 
\end{thm}

In fact, the proof (discussed in Section \ref{subsec:The-case beta pos}
below) shows that the convergence holds at an exponential rate, in
the sense of large deviation theory: for any given $\epsilon>0$ there
exists a positive constant $C_{\epsilon}$ such that 
\[
\text{Prob}\left(d\left(\frac{1}{N}\sum_{i=1}^{N}\delta_{x_{i}},dV_{KE}\right)>\epsilon\right)\leq C_{\epsilon}e^{-N\epsilon},
\]
 where $d$ denotes any metric on the space $\mathcal{P}(X)$ of probability
measures on $X$ compatible with the weak topology. The convergence
in probability implies, in particular, that the measures $dV_{k}$
on $X,$ defined by the expectations $\E(\delta_{N_{k}})$ of the
empirical measure $\delta_{N_{k}}$ converge towards $dV_{KE}$ in
the weak topology of measures on $X:$ 
\[
dV_{k}:=\E(\delta_{N_{k}})=\int_{X^{N_{k}-1}}\mu^{(N_{k})}\rightarrow dV_{KE},\,\,\,k\rightarrow\infty
\]
For $k$ sufficiently large (ensuring that $kK_{X}$ is very ample)
the measures $dV_{k}$ are, in fact, volume forms on $X$ and induce
a sequence of canonical Kähler metrics $\omega_{k}$ on $X,$ expressed
in terms of period type integrals: 

\begin{equation}
\omega_{k}:=\frac{i}{2\pi}\partial\bar{\partial}\log dV_{k}=\frac{i}{2\pi}\partial\bar{\partial}\log\int_{X^{N_{k}-1}}\left|\det S^{(k)}\right|^{2/k},\label{eq:def of omega k KX pos intro}
\end{equation}
 whose integrands are encoded by the degree $k$ component of the
canonical ring of $X.$ The convergence above also implies that the
canonical Kähler metrics $\omega_{k}$ converge, as $k\rightarrow\infty,$
towards the Kähler-Einstein metric $\omega_{KE}$ on $X,$ in the
weak topology. 

\subsubsection{\label{subsec:The-case-beta neg intro}The case $\beta<0$ }

When $-K_{X}>0$, i.e. $X$ is a Fano manifold, there are obstructions
to the existence of a Kähler-Einstein metric. According to the \emph{Yau-Tian-Donaldson
conjecture (YTD)} $X$ admits a Kähler-Einstein metric iff $X$ is
\emph{K-polystable.} The non-singular case was settled in \cite{c-d-s}
and the singular case in \cite{li1,ltw,l-x-z}, building on the proof
of the uniform version of the YTD conjecture on Fano manifolds in
\cite{bbj} (the ``only if'' direction was previously shown in \cite{berman6ii}).
In the probabilistic approach a different type of stability condition
naturally appears, dubbed \emph{Gibbs stability}  (connections with
the YTD conjecture are discussed in \cite{ber12}). The starting point
for the probabilistic approach on a Fano manifold, introduced in \cite[Section 6]{berm8 comma 5},
is the observation that when $-K_{X}>0$ one can replace $k$ with
$-k$ in the previous constructions concerning the case $K_{X}>0$.
Thus, given a positive integer $k$ we set
\[
N_{k}:=\dim H^{0}(X,-kK_{X})
\]
 (which tends to infinity as $k\rightarrow\infty,$ since $-K_{X}$
is ample) and define a measure on $X^{N_{k}}$ by

\begin{equation}
\mu^{(N_{k})}:=\frac{1}{\mathcal{Z}_{N_{k}}}\left|\det S^{(k)}\right|^{-2/k},\,\,\,\,\mathcal{Z}_{N_{k}}:=\int_{X^{N_{k}}}\left|\det S^{(k)}\right|^{-2/k}\label{eq:canon prob measure Fano intro}
\end{equation}
However, in this case it may happen that the normalizing constant$\mathcal{Z}_{N_{k}}$
diverges, since the integrand of $\mathcal{Z}_{N_{k}}$ blows-up along
the zero-locus in $X^{N_{k}}$ of $\det S^{(k)}.$ Accordingly, a
Fano manifold $X$ is called\emph{ Gibbs stable at level $k$} if
$Z_{N_{k}}<\infty$ and \emph{Gibbs stable} if it is Gibbs stable
at level $k$ for $k$ sufficiently large. For a Gibbs stable Fano
manifold $X$ the measure $\mu^{(N_{k})}$ in formula \ref{eq:canon prob measure Fano intro}
defines a canonical symmetric probability measure on $X^{N_{k}}.$
We thus arrive at the following probabilistic analog of the Yau-Tian-Donaldson
conjecture posed in \cite[Section 6]{berm8 comma 5}:
\begin{conjecture}
\label{conj:Fano with triv autom intr}Let $X$ be Fano manifold.
Then 
\begin{itemize}
\item $X$ admits a unique Kähler-Einstein metric $\omega_{KE}$ if and
only if $X$ is Gibbs stable. 
\item If $X$ is Gibbs stable, the empirical measures $\delta_{N}$ of the
corresponding canonical point processes converge in probability towards
the normalized volume form of $\omega_{KE}.$ 
\end{itemize}
\end{conjecture}

In order to briefly compare with the YTD conjecture denote by $\text{Aut \ensuremath{(X)_{0}}}$
the Lie group of automorphisms (biholomorphisms) of $X$ homotopic
to the identity $I.$ Fano manifolds are divided into the two classes,
according to whether $\text{Aut \ensuremath{(X)_{0}}}$ is\emph{ trivial}
or\emph{ non-trivial,} 
\[
\text{Aut}\ensuremath{(X)_{0}=\{I\}\,\,\,\text{or\,\,\,}\text{Aut}\ensuremath{(X)_{0}\neq\{I\}}}
\]
In the former case the Kähler-Einstein metric is uniquely determined
(when it exists), while in the latter case it is only uniquely determined
modulo the action of the group $\text{Aut \ensuremath{(X)_{0}}}.$
This dichotomy is also reflected in the difference between \emph{K-polystability
}and the stronger notion of \emph{K-stability}, which\emph{ }implies
that $\text{Aut \ensuremath{(X)_{0}}}$ is trivial. Similarly, the
Gibbs stability of $X$ also implies that the group $\text{Aut \ensuremath{(X)_{0}}}$
is trivial \cite{ber14} and should thus be viewed as the analog of
K-stability. Accordingly, we shall focus on the case when $\text{Aut \ensuremath{(X)_{0}}}$
is trivial (but see \cite[Conj 3.8]{berm11} for a generalization
of Conjecture \ref{conj:Fano with triv autom intr} to the case when
$\text{Aut \ensuremath{(X)_{0}}}$ is non-trivial).

There is also a natural analog of the stronger notion of \emph{uniform
K-stability} (discussed in more detail in \cite{ber12}). To see this
first recall that Gibbs stability can be given a purely algebro-geometric
formulation, saying that the $\Q-$divisor $\mathcal{D}_{N_{k}}$
in $X^{N_{k}}$ cut out by the (multi-valued) holomorphic section
$(\det S^{(k)})^{1/k}$ of $-K_{X^{N_{k}}}$has mild singularities
in the sense of the Minimal Model Program (MMP). More precisely, $X$
is Gibbs stable at level $k$ iff $\mathcal{D}_{N_{k}}$ is\emph{
klt (Kawamata Log Terminal).} This means that the \emph{log canonical
threshold (lct) }of $\mathcal{D}_{N_{k}}$ satisfies 
\begin{equation}
\text{lct \ensuremath{(\mathcal{D}_{N_{k}})>1}}\label{eq:lct of D N k strictly greater than one}
\end{equation}
(as follows directly from the analytic representation of the log canonical
threshold of a $\Q-$divisor $\mathcal{D},$ recalled in the appendix).
Accordingly, $X$ is called\emph{ uniformly Gibbs stable }if  there
exists $\epsilon>0$ such that, for $k$ sufficiently large, 
\begin{equation}
\text{lct \ensuremath{(\mathcal{D}_{N_{k}})>1+\epsilon}}.\label{eq:lct of D N k strictly greater than one plus eps}
\end{equation}
One is thus led to pose the following purely algebro-geometric conjecture:
\begin{conjecture}
\label{conj:unif Gibbs iff unif K}Let $X$ be a Fano manifold. Then
$X$ is (uniformly) K-stable iff $X$ is (uniformly) Gibbs stable.
\end{conjecture}

The uniform version of the ``if'' direction was settled in \cite{f-o},
using algebro-geometric techniques (see also \cite{ber13} for a different
direct analytic proof that uniform Gibbs stability implies the existence
of a unique Kähler-Einstein metric). However, the converse is still
widely open. And even if confirmed it is a separate analytic problem
to prove the convergence towards the Kähler-Einstein metric in Conjecture
\ref{conj:Fano with triv autom intr}. In \cite[Section 7]{berm11}
a variational approach to the convergence problem was introduced,
which reduces the proof of the convergence towards the volume form
$dV_{KE}$ of Kähler-Einstein metric to establishing the following
convergence result for the normalization constants $\mathcal{Z}_{N_{k}}:$
\begin{equation}
\lim_{N_{k}\rightarrow\infty}-\frac{1}{N_{k}}\log\mathcal{Z}_{N_{k}}=\inf_{\mu\in\mathcal{P}(X)}F(\mu),\label{eq:conv of part intro}
\end{equation}
 where $F(\mu)$ is a functional on the space \emph{$\mathcal{P}(X)$
}of probability measures on $X,$ minimized by $dV_{KE},$ which may
be identified with the Mabuchi functional (see Section \ref{rem:twisted-K-energy}).
This variational approach is inspired by a statistical mechanical
formulation where $F$ appears as a free energy type functional and
$\beta$ appears as the ``inverse temperature''. A central role
is played by  the \emph{partition function }
\begin{equation}
\mathcal{Z}_{N_{k}}(\beta):=\int_{X^{N_{k}}}\left\Vert \det S^{(k)}\right\Vert ^{2\beta/k}dV^{\otimes N_{k}},\,\,\beta\in[-1,\infty[\label{eq:def of Z N beta intro}
\end{equation}
coinciding with the normalization constant $\mathcal{Z}_{N}$ when
$\beta=-1.$\emph{ }However, for $\beta\neq-1$ $\mathcal{Z}_{N_{k}}(\beta)$
depends on the choice of an Hermitian metric $\left\Vert \cdot\right\Vert $
on $-K_{X},$ which, in turn, induces a volume form $dV$ on $X.$
In order to establish the convergence \ref{eq:conv of part intro}
two different approaches were put forth in \cite[Section 7]{berm11},
which hinge on establishing either of the following two hypothesis:
\begin{itemize}
\item The ``upper bound hypothesis'' for the mean energy (discussed in
Section \ref{subsec:The-case beta pos})
\item The ``zero-free hypothesis'' (discussed in Section \ref{subsec:The-zero-free-hypothesis}):
\begin{equation}
\mathcal{Z}_{N_{k}}(\beta)\neq0\,\text{on some \ensuremath{N_{k}-}independent neighborhood \ensuremath{\Omega} of \ensuremath{]-1,0]} in \ensuremath{\C}. }\label{eq:zero-free hyp-2}
\end{equation}
\end{itemize}
While originally defined for $\beta\in[-1,\infty[$ the partition
function $\mathcal{Z}_{N_{k}}(\beta)$ extends to a meromorphic function
of $\beta\in\C,$ all of whose poles appear on the negative real axes.
Indeed, by taking a covering of $X$ the function $\mathcal{Z}_{N_{k}}(\beta)$
may be expressed as a sum of functions of the form 

\begin{equation}
Z(\beta):=\int_{\C^{m}}|f|^{2\beta}\Phi d\lambda,\label{eq:def of Z for f and Phi intro}
\end{equation}
 for a holomorphic function $f$ and a Schwartz function $\Phi$ on
$\C^{m}.$ One can then invoke classical general results of Atyiah
and Bernstein for such meromorphic functions $Z(\beta)$ (recalled
in Section \ref{subsec:Archimedean-zeta-functions} of the appendix).
The first negative pole of $\mathcal{Z}_{N_{k}}(\beta)$ is precisely
the negative of the log canonical threshold $\text{lct}\ensuremath{(\mathcal{D}_{N_{k}})}.$
The zero-free hypothesis referred to above demands that there exists
an $N-$independent neighborhood of $]-1,0]$ in $\C$ where $\mathcal{Z}_{N_{k}}(\beta)\neq0.$
As shown in Section \ref{subsec:The-zero-free-hypothesis} the virtue
of this hypothesis is that it allows one to prove the convergence
in formula \ref{eq:conv of part intro} by ``analytically continuing''
the convergence for $\beta>0$ to $\beta=-1.$ In the statistical
mechanics literature this line of argument goes back to the Lee-Yang
theory of phase transitions (see Remark \ref{rem:Lee-Yang}).

\subsection{The partition function $\mathcal{Z}_{N_{k}}(\beta)$ viewed as local
Archimedean zeta function }

From an algebro-geometric perspective the partition function $\mathcal{Z}_{N_{k}}(\beta)$
(formula \ref{eq:def of Z N beta intro}) is an instance of an \emph{Archimedean
zeta function. }More generally, replacing the local field $\C$ and
its standard Archimedean absolute value $\left|\cdot\right|$ with
a local field $F$ and an absolute value $\left|\cdot\right|_{F}$
on $F$ meromorphic functions $Z(\beta)$ as in formula \ref{eq:def of Z for f and Phi intro}
can be attached to any polynomial $f$ defined over the local field
$F.$ Such meromorphic functions are usually called\emph{ local Igusa
zeta functions }\cite{ig}\emph{.} This is briefly recalled in Section
\ref{subsec:Archimedean-zeta-functions} of the appendix. For example,
the Riemann zeta function $\zeta(s)$ may be expressed as an Euler
product over such local meromorphic functions $Z_{p}(s)$ as $p$
ranges over all primes $p,$ i.e. all non-Archimedean places $p$
of the global field $\Q:$ 
\[
\zeta(s):=\sum_{n=1}^{\infty}n^{-s}=\prod_{p}Z_{p}(s),\,\,\,\,\,\,\,\,Z_{p}(s)=\int_{\Q_{p}^{^{\times}}}\left|x\right|_{\Q_{p}}^{s}\Phi_{p}d^{\times}x=\left(1-p^{-s}\right)^{-1}
\]
 where $\Q_{p}$ is the localization of $\Q$ at $p,$ i.e. the $p-$adic
field $\Q_{p},$ endowed with its standard normalized non-Archimedean
absolute value and multiplicative Haar measure $d^{\times}x$ on $\Q_{p}^{^{\times}}$
and $\Phi_{p}$ denotes the $p-$adic Gaussian. This is explained
in Tate's celebrated thesis \cite{ta}, where it shown that the classical
procedure of completing the Riemann zeta function amounts to including
a factor $Z_{p}(s)$ corresponding to the standard Archimedean absolute
value on $\R,$ which is proportional to the Gamma function. \footnote{expressing $d^{\times}x=x^{-1}dx$ reveals that the role of $\beta$
is played by $s-1$; see Section \ref{subsec:Zeta-integral-expressions for the standard} } In this case all the local factors $Z_{p}(s)$ are manifestly non-zero
(while the corresponding global zeta function $\zeta(s)$ does have
zeros). It should, however, be stressed that it is rare that general
local Igusa zeta functions of the form \ref{eq:def of Z for f and Phi intro}
and their zeros can be computed explicitly. Still, one might hope
that the canonical nature of $\mathcal{Z}_{N_{k}}(\beta)$ may facilitate
the situation. One small step in this direction is taken in Section
\ref{sec:Speculations-on-the}, where some intriguing relations between
the partition functions $\mathcal{Z}_{N_{k}}(\beta)$ and the local
L-functions appearing in the Langlands program are pointed out (generalizing
the local factors $Z_{p}(\beta)$ of the Riemann zeta function). In
particular, it is shown that in the simplest case when $X$ is $n-$dimensional
complex projective space and $N_{k}$ is minimal, i.e. $N_{k}=n+1,$
the partition function $\mathcal{Z}_{N_{k}}(\beta)$ can be identified
with a standard local L-function $L_{p}$ attached to the group $GL(n+1,\Q)$
when the place $p$ of the global field $\Q$ is taken to be the one
defined by the complex Archimedean absolute. Accordingly, in this
particular case $\mathcal{Z}_{N_{k}}(\beta)$ has a strong zero-free
property as a consequence of the standard zero-free property of local
L$-$functions. 

\subsection{\label{subsec:Main-new-results}Main new results in the case of log
Fano curves}

Here it will be demonstrated that both approaches discussed above
are successful in one complex dimension, $n=1.$ The only one-dimensional
Fano manifold $X$ is the complex projective line (the Riemann sphere)
and its Kähler-Einstein metrics are all biholomorphically equivalent
to the standard round metric on the two-sphere. But a geometrically
richer situation appears when introducing weighted points (conical
singularities) on the Riemann sphere. From the algebro-geometric point
of view this fits into the standard setting of \emph{log pairs} $(X,\Delta),$
consisting of complex (normal) projective variety $X$ (here assumed
to be non-singular, for simplicity) endowed with a $\Q-$divisor $\Delta$
on $X,$ i.e. a sum of irreducible subvarieties $\Delta_{i}$ of $X$
of codimension one, with coefficients $w_{i}$ in $\Q.$ In this log
setting the role of the canonical line bundle $K_{X}$ is placed by
the \emph{log canonical line bundle} 
\[
K_{(X,\Delta)}:=K_{X}+\Delta
\]
(viewed as a $\Q-$line bundle) and the role of the Ricci curvature
$\mbox{Ric \ensuremath{\omega}}$ of a metric $\omega$ is played
by twisted Ricci curvature $\mbox{Ric \ensuremath{\omega}}-[\Delta],$
where $[\Delta]$ denotes the current of integration defined by $\Delta.$
The corresponding \emph{log Kähler-Einstein equation} thus reads
\begin{equation}
\mbox{Ric \ensuremath{\omega}}-[\Delta]=\beta\omega,\,\,\,\beta=\pm1\label{eq:log KE eq with beta intro}
\end{equation}
 where $[\Delta]$ denotes the current of integration along $\Delta.$
When $\beta$ is non-zero existence of a solution $\omega_{KE}$ forces
\[
\beta(K_{X}+\Delta)>0
\]
In general, the equation \ref{eq:log KE eq with beta intro} should
be interpreted in the weak sense of pluripotential theory \cite{egz,bbegz}.
However, in case when $(X,\Delta)$ is\emph{ log smooth, }i.e. the
components of $\Delta$ have simple normal crossings (which means
that they intersect transversally) it follows from \cite{jmr,g-p}
that a positive current $\omega$ solves the equation \ref{eq:log KE eq with beta intro}
iff $\omega$ is a bona fide Kähler-Einstein metric on $X-\Delta$
and $\omega$ has edge-cone singularities along $\Delta,$ with cone-angle
$2\pi(1-w_{i}),$ prescribed by the coefficients $w_{i}$ of $\Delta.$
In particular, in the \emph{orbifold case
\begin{equation}
\Delta=\sum(1-\frac{1}{m_{i}})\Delta_{i},\,\,m_{i}\in\Z_{+}\label{eq:orbifold case}
\end{equation}
}the log Kähler-Einstein metrics locally lifts to a bona fide Kähler-Einstein
metric on local coverings of $X$ (branched along $\Delta$ and $K_{X}+\Delta$
may be identified with the orbifold canonical line bundle) \cite[Section 2]{b-g-k}. 
\begin{example}
Let $X$ be the complex hypersurface of weighted projective space
$\P(a_{0},...,a_{n}),$ cut out by a quasi-homogeneous polynomial
$F$ on $\C^{n+1}$ of degree $d,$ whose zero-locus $Y\subset\C^{n+1}-\{0\}$
is assumed non-singular. Then the orbifold $(X,\Delta)$ defined by
the branching divisor $\Delta$ on $X$ of the fibration $Y-\{0\}\rightarrow X,$
induced by the natural quotient projection $\C^{n+1}-\{0\}\rightarrow\P(a_{0},...,a_{n}),$
is a Fano orbifold (i.e. $-(K_{X}+\Delta)>0$) iff $d<a_{0}+a_{1}+...+a_{n}.$
\end{example}

The probabilistic approach naturally extends to the setting of log
pairs $(X,\Delta)$ satisfying $\beta(K_{X}+\Delta)>0$ yielding a
canonical probability measure on $X^{N_{k}}$, that we shall denote
by $\mu_{\Delta}^{(N_{k})}.$ Indeed, one simply replaces the canonical
line bundle $K_{X}$ with the log canonical line bundle $K_{(X,\Delta)}$
in the previous constructions (cf. \cite[Section 5]{berm8 comma 5}
and \cite[Section 3.2.4]{berm11}). 

\subsubsection{Log Fano curves}

Let now $(X,\Delta)$ be a log Fano curve $(X,\Delta),$ i.e. $X$
is the complex projective line and 
\[
\Delta=\sum_{i=1}^{m}w_{i}p_{i}
\]
for positive weights $w_{i}$ satisfying $\sum_{i=1}^{m}w_{i}<2.$
In this case it turns out that the ``upper bound hypothesis'' for
the mean energy does hold, which leads to the following result announced
in \cite[Section 3.2.4]{berm11}:
\begin{thm}
\label{thm:log curve intro}Let $(X,\Delta)$ be a log Fano curve.
Then the following is equivalent
\begin{itemize}
\item $(X,\Delta)$ is  Gibbs stable
\item $(X,\Delta)$ is uniformly Gibbs stable
\item The following weight condition holds: 
\begin{equation}
w_{i}<\sum_{i\neq j}w_{j},\,\,\,\forall i\label{eq:weight condition intro}
\end{equation}
\item There exists a unique a unique Kähler-Einstein metric $\omega_{KE}$
for $(X,\Delta)$
\end{itemize}
Moreover, if any of the conditions above hold, then the laws of the
corresponding empirical measures $\delta_{N}$ satisfy a Large Deviation
Principle (LDP) with speed $N,$ whose rate functional has a unique
minimizer, namely $\omega_{KE}/\int_{X}\omega_{KE}.$ In particular,
for any given $\epsilon>0,$ 
\[
\text{Prob}\left(d\left(\frac{1}{N}\sum_{i=1}^{N}\delta_{x_{i}},\frac{\omega_{KE}}{\int_{X}\omega_{KE}}\right)>\epsilon\right)\leq C_{\epsilon}e^{-N\epsilon}.
\]
\end{thm}

Existence of solutions to the log Kähler-Einstein equation \ref{eq:log KE eq with beta intro}
in the one-dimensional setting were first shown in \cite{tr}, under
the weight condition \ref{eq:weight condition intro} and uniqueness
in \cite{l-t}. The weight condition \ref{eq:weight condition intro}
is also equivalent to uniform K-stability of $(X,\Delta)$ \cite[Ex. 6.6]{fu2}
and thus the previous theorem confirms Conjecture \ref{conj:unif Gibbs iff unif K}
for log Fano curves. 

We also show that in the case when the support of $\Delta$ consists
of three points the following variant of the ``zero-free hypothesis''
holds:
\[
\mathcal{Z}_{N_{k},\Delta}\neq0
\]
 when the coefficients of $\Delta$ are complexified, so that $\mathcal{Z}_{N_{k},\Delta}$
is extended to a meromorphic function on $\C^{3}$ (the proof exploits
that $\mathcal{Z}_{N_{k},\Delta}$ can be expressed as the complex
Selberg integral, which first appeared in Conformal Field Theory).
This leads to an alternative proof of the previous theorem, in this
particular case, by ``analytically continuing'' the convergence
result in the case $K_{X}+\Delta>0$ to the log Fano case $K_{X}+\Delta<0.$ 
\begin{example}
The case of three points includes, in particular, the case when $X$
is a \emph{Fano orbifold curve. }Such a curve may be embedded into
a weighted $\P^{2}$ and is defined by the zero-locus of explicit
quasi-homogeneous polynomial $F(X,Y,Z)$ in $\C^{3}$(the du Val singularities).
In the case of three orbifold points there always exists a unique
log Kähler-Einstein metric on $X,$ concretely realized as the quotient
$\P^{1}/G$ of the standard $SU(2)-$invariant metric on $\P^{1}$
under the action of a discrete subgroup $G$ of $SU(2)$ (branched
over the three points in question).
\end{example}

\subsection{Acknowledgments}

I am grateful to Sébastien Boucksom, Dennis Eriksson, Mattias Jonsson,
Yuji Odaka, Daniel Persson and Yanir Rubinstein for stimulating discussions.
Also thanks to the referee whose comments helped to improve the exposition.
This work was supported by grants from the Knut and Alice Wallenberg
foundation, the Göran Gustafsson foundation and the Swedish Research
Council. 

\subsection{Organization}

In section \ref{sec:Conditional-convergence-results} conditional
convergence results on log Fano varieties are obtained, formulated
in terms of either the ``upper bound hypothesis'' on the mean-energy
or the ``zero-free hypothesis'' of the partition function. Then
- after a digression on the Calabi-Yau equation in Section \ref{sec:Intermezzo:-a-zero-free}
- in Section \ref{sec:A-case-study:} the hypotheses in question are
verified for log Fano curves and Fano orbifolds, respectively. Section
\ref{sec:Speculations-on-the} is of a speculative nature, comparing
the strong form of the zero-free hypothesis with the standard zero-free
property of the local L-functions appearing in the Langlands program.
The paper is concluded with an appendix, providing background on log
canonical thresholds and Archimedean zeta functions.

\section{\label{sec:Conditional-convergence-results}Conditional convergence
results on log Fano varieties}

In this section it is explained how to reduce the proof of the convergence
on Fano manifolds $X$ in Conjecture \ref{conj:Fano with triv autom intr}
to establishing either one of two different hypotheses, building on
\cite[Section 7]{berm11}. More generally, we will consider the setup
of log Fano varieties $(X,\Delta),$ discussed in Section \ref{subsec:Main-new-results}).
For simplicity $X$ will be assumed to be non-singular. We will be
using the standard correspondence between metrics $\left\Vert \cdot\right\Vert $
on log canonical line bundles $-(K_{X}+\Delta)$ and volume forms
$dV_{\Delta}$ on $X-\Delta,$ which are singular when viewed as measures
on $X$ (see \cite[Section 4.1.7]{berm11} for background, where the
measure $dV_{\Delta}$ is denoted by $\mu_{0}$).

\subsection{Setup}

Let $(X,\Delta)$ be a log Fano variety.  As recalled in Section \ref{subsec:Main-new-results}
this means that $\Delta$ is a divisor with positive coefficients,
and $-(K_{X}+\Delta)>0.$ We will allow $\Delta$ to have real coefficients.
Set
\[
N_{k}:=\dim H^{0}\left((X,-k(K_{X}+\Delta)\right),
\]
 where $k$ ranges over the positive numbers with the property that
$-k(K_{X}+\Delta)$ is a well-defined line bundle on $X.$ To simplify
the notation we will often drop the subscript $k$ in the notation
for $N_{k}.$ Since, 
\[
k\rightarrow\infty\iff N\rightarrow\infty,
\]
 this should, hopefully, not cause any confusion. As discussed in
Section \ref{subsec:Main-new-results}, assuming that $(X,\Delta)$
is Gibbs stable we get a sequence of canonical probability measures
$\mu_{\Delta}^{(N)}$ on $X^{N}$. Fixing a smooth Hermitian metric
$\left\Vert \cdot\right\Vert $ on the $\R-$line bundle $-(K_{X}+\Delta)$
with positive curvature $\mu_{\Delta}^{(N)}$ may be expressed as
\begin{equation}
\mu_{\Delta}^{(N)}:=\frac{1}{\mathcal{Z}_{N}}\left\Vert \det S^{(k)}\right\Vert ^{2/k}dV_{(X,\Delta)}^{\otimes N},\,\,\,\mathcal{Z}_{N}:=\int_{X^{N}}\left\Vert \det S^{(k)}\right\Vert ^{2/k}dV_{(X,\Delta)}^{\otimes N},\label{eq:mu N Delta in terms of metric}
\end{equation}
where $dV_{(X,\Delta)}$ is the singular volume form on $X$ corresponding
to the metric $\left\Vert \cdot\right\Vert $ on $-(K_{X}+\Delta)$
and $\det S^{(k)}$ is the Slater determinant of $H^{0}\left((X,-k(K_{X}+\Delta)\right)$
induced by a choice of bases $s_{1}^{(k)},...s_{N}^{(k)}$ for $H^{0}(X,-k(K_{X}+\Delta)),$
defined as in formula \ref{eq:slater determinant}. Since $\mu_{\Delta}^{(N)}$
is independent of the choice of bases we may as well assume that the
basis is orthonormal with respect to the Hermitian product induced
by $(\left\Vert \cdot\right\Vert ,dV).$ The condition that $(X,\Delta)$
is Gibbs stable means that the normalization constant $\mathcal{Z}_{N}$
is finite. Hence, it implies that the local densities of $dV$ are
in $L_{loc}^{1}$ ( which in algebraic terms means that $\Delta$
is klt divisor). 

From a statistical mechanical point of view the probability measure
$\mu_{\Delta}^{(N)}$ on $X^{N}$ may be expressed as the \emph{Gibbs
measure}

\begin{equation}
\mu_{\beta}^{(N)}=\frac{e^{-\beta NE^{(N)}}}{\mathcal{Z}_{N}(\beta)}dV_{\Delta}^{\otimes N},\,\,\,E^{(N)}(x_{1},...,x_{N}):=-\frac{1}{kN}\log\left(\left\Vert \det S^{(k)}(x_{1},...,x_{N})\right\Vert ^{2}\right)\label{eq:Gibbs meas}
\end{equation}
with $\beta=-1.$ In physical terms the Gibbs measure represents the
microscopic state of $N$ interacting particles in thermal equilibrium
at inverse temperature $\beta,$ with $E^{(N)}(x_{1},...,x_{N})$
playing the role of the\emph{ energy per particle and} the normalizing
constant 
\begin{equation}
\mathcal{Z}_{N}(\beta)=\int_{X^{N}}e^{-\beta NE^{(N)}}dV_{(X,\Delta)}^{\otimes N}=\int_{X^{N}}\left\Vert \det S^{(k)}\right\Vert ^{2\beta/k}dV_{(X,\Delta)}^{\otimes N}\label{eq:def of part function text II}
\end{equation}
 is called the \emph{partition function.} It should, however, be stressed
that, while the probability measure $\mu_{\Delta}^{(N)}$ is canonical,
i.e. independent of the choice of metric $\left\Vert \cdot\right\Vert $
, this is not so when $\beta\neq-1.$ But one advantage of introducing
the parameter $\beta$ is that $\mu_{\beta}^{(N_{k})}$ is a well-defined
probability measure as long as $\beta>-\text{lct \ensuremath{(X,\Delta)}},$
where $\text{lct \ensuremath{(X,\Delta)}}$ denotes the global log
canonical threshold of $(X,\Delta)$ (whose definition is recalled
in the appendix). In particular, it is, trivially, well-defined when
$\beta>0.$

Fixing $\beta\in[-1,\infty[$ we can can view the empirical measure
\[
\delta_{N}:=\frac{1}{N}\sum_{i=1}^{N}\delta_{x_{i}}:\,\,\,X^{N}\rightarrow\mathcal{P}(X)
\]
 as a random discrete measure on $X.$ To be more precise: $\delta_{N}$
is a random variable on the ensemble $\left(X^{N},\mu_{\beta}^{(N)}\right),$
taking values in the space $\mathcal{P}(X)$ of probability measures
on $X.$ Accordingly, the\emph{ law} of $\delta_{N}$ is the probability
measure
\[
\Gamma_{N,\beta}:=(\delta_{N})_{*}\mu_{\beta}^{(N)}\in\mathcal{P}\left(\mathcal{P}(X)\right)
\]
 om $\mathcal{P}(X),$ defined as  the push-forward of the probability
measure $\mu_{\beta}^{(N)}$ on $X^{N}$ to $\mathcal{P}(X)$ under
the map $\delta_{N}.$ 

\subsection{\label{subsec:The-case beta pos}The case $\beta>0$}

The following result, which is a special case of \cite[Thm 5.7]{berm8}
(when $\Delta$ is trivial) and \cite[Thm 4.3]{berm8 comma 5} (when
$\Delta$ is non-trivial) establishes a Large Deviation Principle
(LDP) for the laws $\Gamma_{N,\beta}$ of $\delta_{N}$ as $N\rightarrow\infty,$
which may be symbolically expressed as 
\[
\Gamma_{N,\beta}:=(\delta_{N})_{*}\mu_{\beta}^{(N)}\sim e^{-N\left(F(\mu)-F(\beta)\right)},\,\,\,N\rightarrow\infty
\]
 (formally viewing the right hand side as a density on the infinite
dimensional space $\mathcal{P}(X);$ the precise meaning of the LDP
is recalled below).
\begin{thm}
\label{thm:beta pos}Let $(X,\Delta)$ be a log Fano variety. For
$\beta>0$ the sequence $\Gamma_{N,\beta}$ of probability measures
on $\mathcal{P}(X)$ satisfies a LDP speed $N$ and rate functional
\begin{equation}
F_{\beta}(\mu)-F(\beta),\,\,\,F(\mu):=\beta E(\mu)+\text{Ent}(\mu),\,\,\,\,F(\beta):=\inf_{\mathcal{P}(X)}F_{\beta}(\mu),\label{eq:def of rate functional}
\end{equation}
 where $E(\mu)$ is the pluricomplex energy of $\mu$ relative to
the Kähler form $\omega$ defined by the curvature of the metric $\left\Vert \cdot\right\Vert $
on $-(K_{X}+\Delta)$ and $\text{Ent}(\mu)$ is the entropy of $\mu$
relative to $dV_{\Delta}.$ In particular, the random measure $\delta_{N}$
converges in probability, as $N\rightarrow\infty,$ to the unique
minimizer $\mu_{\beta}$ of $F_{\beta}$ in $\mathcal{P}(X),$ i.e.
\begin{equation}
\lim_{N\rightarrow\infty}\Gamma_{N,\beta}=\delta_{\mu_{\beta}}\,\,\text{in \ensuremath{\mathcal{P}\left(\mathcal{P}(X)\right)}}\label{eq:conv of laws in theorem}
\end{equation}
and the following convergence of the partition functions $\mathcal{Z}_{N}(\beta)$
holds
\begin{equation}
\lim_{N\rightarrow\infty}-\frac{1}{N}\log\mathcal{Z}_{N}(\beta)=F(\beta).\label{eq:conv of partit function in theorem}
\end{equation}
 
\end{thm}

We recall that the \emph{entropy} $\text{Ent}(\mu)$ of $\mu$ relative
to a given measure $\nu$ is defined by 
\[
\text{Ent}(\mu)=\int_{X}\log\frac{\mu}{\nu}\mu
\]
 when $\mu$ has a density with respect to $\nu$ and otherwise $\text{Ent}(\mu):=\infty.$
\footnote{We are using the ``mathematical'' sign convention for the entropy,
which renders $\text{Ent}(\mu)$ non-negative when the reference measure
$\nu$ is a probability measure and thus $\text{Ent}(\mu)$ coincides
with the \emph{Kullback--Leibler divergence }in information theory. } As for the pluricomplex energy $E(\mu)$ of a measure $\mu$ on $X,$
relative to a reference form $\omega_{0},$ it was first introduced
in \cite[Thm 4.3]{bbgz}. From a thermodynamical point of view the
functional $F_{\beta}(\mu),$ introduced in \cite[Thm 4.3]{berm6},
can be viewed as the\emph{ free energy}\footnote{Strictly speaking it is $F_{\beta}/\beta$ which plays the role of
free energy in thermodynamics}\emph{.} The pluricomplex $E(\mu)$ may be defined as the greatest
lsc extension to $\mathcal{P}(X)$ of the functional $E(\mu)$ on
the space of volume forms $\mu$ in $\mathcal{P}(X)$ whose first
variation is given by 
\begin{equation}
dE(\mu)=-\varphi_{\mu}.\label{eq:first var of E}
\end{equation}
 where $\varphi_{\mu}$ is a smooth solution to the complex Monge-Ampère
equation (also known as the Calabi-Yau equation):
\[
\frac{1}{V}(\omega+\frac{i}{2\pi}\partial\bar{\partial}\varphi_{\beta})^{n}=\mu,\,\,\,V:=\int_{X}\omega^{n}.
\]
This property determines the functional $E(\mu)$ up to an additive
constant which is fixed by imposing the normalization condition 
\begin{equation}
E(\omega_{0}^{n}/V)=0,\label{eq:E vanishes on reference vol}
\end{equation}
in the case when the reference form $\omega_{0}$ is Kähler. Using
the property \ref{eq:first var of E} it is shown in \cite[Prop 4.1]{berm11}
that the minimizer $\mu_{\beta}$ of $F_{\beta}(\mu)$ is the normalized
volume form on $X-\Delta$ uniquely determined by the property that
\[
\mu_{\beta}=e^{\beta\varphi_{\beta}}dV_{\Delta},
\]
 where the function $\varphi_{\beta}$ is the unique smooth bounded
Kähler potential on $X-\Delta$ solving the complex Monge-Ampère equation

\begin{equation}
\frac{1}{V}(\omega+\frac{i}{2\pi}\partial\bar{\partial}\varphi_{\beta})^{n}=e^{\beta\varphi_{\beta}}dV_{\Delta}.\label{eq:complex ma eq with beta}
\end{equation}
 It follows that the corresponding Kähler form 
\[
\omega_{\beta}:=\omega+\frac{1}{\beta}\frac{i}{2\pi}\partial\bar{\partial}\log\frac{\mu_{\beta}}{dV_{\Delta}}\,\left(=\omega+\frac{i}{2\pi}\partial\bar{\partial}\varphi_{\beta}\right)
\]
satisfies the twisted Kähler-Einstein equation 
\begin{equation}
\mbox{\ensuremath{\mbox{Ric}}\ensuremath{\,\omega_{\beta}}}-[\Delta]=-\beta\omega_{\beta}+(\beta+1)\omega_{0},\label{eq:twisted KE with beta}
\end{equation}
on $X,$ coinciding with the (log) Kähler-Einstein equation \ref{eq:log KE eq with beta intro}
when $\beta=-1.$ 
\begin{rem}
\label{rem:twisted-K-energy}Incidentally, the functional 
\[
\mathcal{M}(\varphi):=F_{-1}\left(\frac{1}{V}(\omega+\frac{i}{2\pi}\partial\bar{\partial}\varphi_{\beta})^{n}\right)
\]
 coincides with the\emph{ Mabuchi functional} for the log Fano variety
$(X,\Delta),$ as explained in \cite[Section 5.3]{berm11}. Moreover,
the twisted Kähler-Einstein equation \ref{eq:twisted KE with beta}
coincides with the logarithmic version of Aubin's continuity equation
with ``time-parameter'' $t:=-\beta.$
\end{rem}

The precise definition of a LDP, which goes back to Cramér and Varadhan
\cite{d-z}, is recalled in \cite[Prop 4.1]{berm11}. For the purpose
of the present paper it will be convenient to use the following equivalent
(``dual'') characterization of the LDP in the previous theorem:
for any continuous function $\Phi(\mu)$ on $\mathcal{P}(X):$
\begin{equation}
\lim_{N\rightarrow\infty}-\frac{1}{N}\log\int_{X^{N}}e^{-N\beta E^{(N)}}e^{-N\Phi(\delta_{N})}=\inf_{\mathcal{P}(X)}\left(F(\mu)+\Phi(\mu)\right)\label{eq:LDP in terms of Phi}
\end{equation}
 (as follows from well-known general results of Varadhan  and Bryc
\cite[Thm 4.4.2]{d-z}). 

\subsubsection{Outline of the proof }

Before turning to the case when $\beta<0$ we briefly recall that
a key ingredient in the proof of the previous theorem is the convergence
\begin{equation}
E^{(N)}(x_{1},...,x_{N})\rightarrow E(\mu),\,\,N\rightarrow\infty,\label{eq:Gamma conv}
\end{equation}
 which hold in the sense of Gamma$-$convergence (deduced from the
convergence and differentiability of weighted transfinite diameters
in \cite[Thm A, Thm B]{b-b}). Combining this convergence with some
heuristics going back to Boltzmann suggests that the contribution
of the volume form $dV^{\otimes N}$ in the Gibbs measure \ref{eq:Gibbs meas}
should give rise to the additional entropy term appearing in the rate
functional: 
\[
(\delta_{N})_{*}\left(e^{-\beta NE^{(N)}}dV^{\otimes N}\right)\sim e^{-NE(\mu)}(\delta_{N})_{*}\left(dV^{\otimes N}\right)\sim e^{-N\beta E(\mu)}e^{-N\text{Ent}(\mu)}
\]
This is made rigorous in \cite{berm8} using an effective submean
property of the density of $\mu_{\beta}^{(N)}$ on the $N-$fold symmetric
product of $X,$ viewed as a Riemannian orbifold (leveraging results
in geometric analysis).

\subsection{The case $\beta<0$\label{subsec:The-case beta neg}}

In the case when $\beta<0$ we may define the free energy functional
$F_{\beta}(\mu)$ by the same expression as in formula \ref{eq:def of rate functional},
$F_{\beta}=\beta E+\text{Ent}(\mu),$ when $E_{\omega_{0}}(\mu)<\infty$
and otherwise we set $F_{\beta}(\mu)=\infty.$ The definition is made
so that we still have $F_{\mu}(\mu)\in]-\infty,\infty]$ with $F_{\mu}(\mu)<\infty$
iff both $E(\mu)<\infty$ and Ent$(\mu)<\infty.$

In order to handle the large $N-$limit in the case when $\beta<0$
a variational approach was introduced in \cite[Section 7]{berm11},
which reduces the problem to establishing the following \emph{``upper
bound hypothesis}'' for the mean energy:

\begin{equation}
\limsup_{N\rightarrow\infty}\int_{X^{N}}E^{(N)}\mu_{\Delta,\beta}^{(N)}\leq E(\Gamma_{\beta}):=\int_{\mathcal{P}(X)}E(\mu)\Gamma_{\beta}(\mu)\label{eq:upper bound property}
\end{equation}
for any large $N-$limit point $\Gamma$ of $\Gamma_{N,\beta}$ in
$\mathcal{X}.$  This property is independent of the choice of metric
$\left\Vert \cdot\right\Vert $ on $-(K_{X}+\Delta).$ Moreover the
corresponding lower bound always holds (as follows from the convergence
\ref{eq:Gamma conv}). The following theorem is an extension of the
results in \cite[Section 7]{berm11} to the case when $\Delta$ is
non-trivial.
\begin{thm}
\label{thm:1 implies 2 implies 3}Let $(X,\Delta)$ be a  log Fano
variety. Assume that $(X,\Delta)$ is uniformly Gibbs stable. Then
$(X,\Delta)$ admits a unique Kähler-Einstein metric $\omega_{KE}.$
Moreover, in the following list each statement implies the next one:
\begin{enumerate}
\item The ``upper bound hypothesis'' \ref{eq:upper bound property} for
the mean energy holds when $\beta=-1$
\item The convergence \ref{eq:conv of partit function in theorem} for the
partition functions holds when $\beta=-1$
\item The empirical measures $\delta_{N}$ of the canonical random point
process on $X$ converge in law towards the normalized volume form
$dV_{KE}$ of $\omega_{KE},$ i.e. the convergence \ref{eq:conv of laws in theorem}
holds when $\beta=-1.$ 
\end{enumerate}
Furthermore,if the ``upper bound hypothesis'' \ref{eq:upper bound property}
is replaced by the stronger hypothesis that the convergence holds
when $E^{(N)}$ is replaced by $E^{(N)}+\Phi(\delta_{N})$ for any
continuous functional $\Phi$ on $\mathcal{P}(X)$ (and $E$ is replaced
by $E+\Phi)$ then the LDP hold in Theorem \ref{thm:beta pos} holds
for $\beta=-1.$ 
\end{thm}

\begin{proof}
The proof in the general case is similar to the case when $\Delta$
is trivial. Indeed, the assumption that $(X,\Delta)$ is uniformly
Gibbs stable implies, by a simple modification of the proof of \cite[Thm 2.5]{f-o}
(concerning the case when $\Delta$ is trivial) that $\delta(X,\Delta)>1,$
which by \cite{fu2} is equivalent to $(X,\Delta)$ being is uniformly
K-stable. Hence, by the solution of the uniform version of the YTD
conjecture for log Fano varieties $(X,\Delta)$ with $X$ non-singular
in \cite{bbj} (extended to general log Fano varieties in \cite{ltw,li1})
it follows that $(X,\Delta)$ admits a unique Kähler-Einstein metric.
Next we summarize the proof of the convergence in \cite[Section 7]{berm11};
all steps are essentially the same in the case when $\Delta$ is non-trivial.
Set
\begin{equation}
F_{N}(\beta):=-\frac{1}{N}\log\mathcal{Z}_{N}(\beta),\,\,\,F(\beta):=\inf_{\mathcal{\mu\in P}(X)}F_{\beta}(\mu)\label{eq:def of F N beta}
\end{equation}
 and consider the mean free energy functional on $\mathcal{P}(X^{N})$
defined by
\[
F_{N}(\mu_{N}):=\beta\int_{X^{N}}E^{(N)}\mu_{N}+\frac{1}{N}\text{Ent}(\mu_{N}),
\]
where $\text{Ent}(\mu_{N})$ denotes the entropy of $\mu_{N}$ relative
to $(dV_{\Delta})^{\otimes N}.$ By Gibbs variational principle (or
Jensen's inequality) 
\begin{equation}
F_{N}(\beta)=\inf_{\mu_{N}\in\mathcal{P}(X^{N})}F_{N,\beta}(\mu_{N})=F_{N,\beta}(\mu_{N,\beta}).\label{eq:Gibbs V P}
\end{equation}
Moreover, 
\begin{equation}
F(\beta)=\inf_{\mathcal{P}\left(\mathcal{P}(X)\right)}F_{\beta}(\Gamma)=F_{\beta}(\delta_{\mu_{\beta}}),\label{eq:var principle for F beta}
\end{equation}
 where $F_{\beta}(\Gamma)$ denotes the following functional on $\mathcal{P}\left(\mathcal{P}(X)\right):$
\[
F_{\beta}(\Gamma):=\int_{\mathcal{P}(X)}F_{\beta}(\mu)\Gamma
\]
and $\delta_{\mu_{\beta}}$ is the unique minimizer of $F(\Gamma)$
in $\mathcal{P}\left(\mathcal{P}(X)\right)$ (using that $F(\mu)$
is lsc, thanks to the energy/entropy compactness theorem in \cite{bbegz}
and hence $F(\Gamma)$ is lsc and linear on $\mathcal{P}\left(\mathcal{P}(X)\right)).$
Now, as shown in the course of the proof of \cite[Thm 6.7]{berm8 comma 5}
(and refined in Step 1 in the proof of \cite[Thm 7.6]{berm11}) for\emph{
any} $\beta,$ the following inequality holds 
\begin{equation}
\limsup_{N\rightarrow\infty}F_{N}(\beta)\leq F(\beta).\label{eq:upper bound on F N in pf}
\end{equation}
(as follows from combining Gibbs variational principle with the Gamma-convergene
\ref{eq:Gamma conv} of $E^{(N)}$ towards $E(\mu)$). Combining Gibbs
variational principle \ref{eq:Gibbs V P} with the variational principle
\ref{eq:var principle for F beta} for $F(\beta)$ this means that
\[
\limsup_{N\rightarrow\infty}\left(\inf_{\mu_{N}\in\mathcal{P}(X^{N})}F_{N,\beta}(\mu_{N})\right)\leq\inf_{\mathcal{\mu\in P}(X)}F_{\beta}(\mu).
\]
Moreover, as shown in \cite[Section 7]{berm11}, if the ``upper bound
hypothesis'' on the mean energy holds, then the corresponding lower
bound also holds, i.e. the convergence \ref{eq:conv of partit function in theorem}
of the partition functions holds: 
\begin{equation}
\lim_{N\rightarrow\infty}F_{N}(\beta)=F(\beta).\label{eq:conv of FN beta}
\end{equation}
Indeed, combining the ``upper bound hypothesis'' with the well-known
sub-additivity property of the mean entropy, yields
\[
F_{\beta}(\Gamma_{\beta})\leq\liminf_{N\rightarrow\infty}F_{N,\beta}(\mu_{N,\beta})
\]
 for any limit point $\Gamma_{\beta}$ of $\Gamma_{N,\beta},$ in
the case $\beta=-1.$ Combined with the upper bound \ref{eq:upper bound on F N in pf}
and formula \ref{eq:var principle for F beta} $F(\beta)$ it then
follows that $\Gamma_{\beta}$ minimizes $F_{-1}(\Gamma)$ and hence,
by the uniqueness of minimizer, $\Gamma=\delta_{\mu_{-1}},$ as desired.
All in all, this shows that $"1\implies2\implies3"$ in the theorem.

Finally, to prove the LDP stated in the theorem one just repeats the
previous argument with $E^{(N)}$ replaced by $E_{\Phi}^{(N)}:=E^{(N)}+\Phi(\delta_{N}).$
Then $\mathcal{Z}_{N}(\beta)$ gets replaced with $\int_{X^{N}}e^{-NE_{\Phi}^{(N)}}dV^{\otimes N}$
and hence the convergence \ref{eq:LDP in terms of Phi} follows, as
before, from the implication $1\implies2,$ now applied to $E_{\Phi}^{(N)}.$

In fact, the implications in the previous theorem may ``almost''
be reversed, by exploiting that that the mean $N-$particular energy
at inverse temperature $\beta$ is proportional to the logarithmic
derivative of $Z_{N}(\beta).$ More precisely, the following theorem
holds, where it is assumed, for technical reasons, that $X$ is a
Fano orbifold. 
\end{proof}
\begin{thm}
\label{thm:equiv cond for conv}Let $(X,\Delta)$ be a Fano orbifold
and assume that $(X,\Delta)$ is uniformly Gibbs stable. Then there
exists $\epsilon>0$ such that $F_{\beta}$ admits a unique minimizer
$\mu_{\beta}$ for any $\beta\in]-1-\epsilon,0[.$ Moreover, the following
is equivalent: 
\begin{enumerate}
\item The ``upper bound hypothesis'' for the mean energy \ref{eq:upper bound property}
holds for any $\beta\in]-1-\epsilon,0[$
\item The convergence \ref{eq:conv of partit function in theorem} for the
partition functions holds for any $\beta\in]-1-\epsilon,0[$
\item The convergence \ref{eq:conv of partit function in theorem} for the
partition functions holds \emph{and} the convergence \ref{eq:conv of laws in theorem}
of the laws of $\delta_{N}$ holds for any $\beta\in]-1-\epsilon,0[.$
\end{enumerate}
Furthermore, If 1,2 or 3 holds, then 
\begin{equation}
\lim_{N\rightarrow\infty}\int_{X^{N}}E^{(N)}\mu_{\Delta,\beta}^{(N)}=E(\mu_{\beta}).\label{eq:conv of mean en}
\end{equation}
\end{thm}

\begin{proof}
First assume that $(X,\Delta)$ is a log Fano variety. As explained
in the proof of the previous theorem $X$ admits a unique Kähler-Einstein
metric. Hence, it follows from \cite{d-r} (and \cite{bbj}) that
$F_{-1}(\mu)$ is coercive with respect to $E,$ i.e. there exists
$\epsilon>0$ such that 
\[
F_{-1}\geq\epsilon E-1/\epsilon
\]
 on $\mathcal{P}(X).$ Thus $F_{\beta}$ is also coercive wrt $E$
for any $\beta>-1-\epsilon.$ In particular, it follows from the energy-entropy
compactness theorem in \cite{bbegz} that $F_{\beta}$ admits a minimizer.
Moreover, as shown in \cite{bbegz} any minimizer has the property
that the corresponding function $\varphi_{\beta}$ satisfies the complex
Monge-Ampère equation \ref{eq:complex ma eq with beta}. Next assume
that $(X,\Delta)$ is a Fano orbifold. Then, for $\beta$ sufficiently
close to $-1$ the equation \ref{eq:complex ma eq with beta} has
a unique solution. Indeed, since the Kähler-Einstein metric is unique
the orbifold $X$ admits no non-trivial orbifold holomorphic vector
fields, which, in turn, implies that the linearization of the equation
\ref{eq:complex ma eq with beta} has a unique solution, defining
a smooth function in the orbifold sense (see \cite{d-k}). It then
follows from a standard application of the implicit function theorem
on orbifolds that the solution $\phi_{\beta}$ is uniquely determined
for $\beta$ sufficiently close to $-1.$ 

By the previous theorem (and its proof) will be enough to show that
$2\implies1.$ Since, trivially, $2\implies3$ we have that $\Gamma_{\beta}=\delta_{\mu_{\beta}}$
and hence it will be enough to show the convergence in formula \ref{eq:conv of mean en}
. To this end first note that the functions $F_{N}(\beta)$ and $F(\beta)$
(defined in formula\ref{eq:def of F N beta}) are concave in $\beta,$
as follows readily from the definitions. Moreover, $F_{N}(\beta)$
and $F(\beta)$ are differentiable on $(]-1-\epsilon,0[$ and
\begin{equation}
\frac{dF_{N}(\beta)}{d\beta}=\int_{X^{N}}E^{(N)}\mu_{\Delta,\beta}^{(N)},\,\,\,\frac{dF(\beta)}{d\beta}=E(\mu_{\beta}),\label{eq:form for deriv av F}
\end{equation}
 using that $\mu_{\beta}$ is the unique minimizer of $F_{\beta}.$
Hence, if the convergence in item $2$ of the theorem holds, then
it follows from basic properties of concave functions that the derivative
of $F_{N}(\beta)$ converges towards the derivative of $F(\beta)$
at $\beta=-1$ (see \cite[Lemma 3.1]{b-b-w}). Applying formula \ref{eq:form for deriv av F}
thus concludes the proof of the convergence \ref{eq:conv of mean en}. 
\end{proof}
\begin{rem}
\label{rem:implicit}The reason that we have assumed that $(X,\Delta)$
is a Fano \emph{orbifold} is that the proof involves the implicit
function theorem in Banach spaces and thus relies on analytic properties
of the linearized log Kähler-Einstein equation. We will come back
to this point in Section \ref{subsec:Deforming-the-divisor}.
\end{rem}

\subsection{\label{subsec:The-zero-free-hypothesis}The zero-free hypothesis }

An alternative approach towards the case $\beta<0$ was also introduced
in \cite[Section 7.1]{berm11}. In a nutshell, it aims to ``analytically
continue'' the convergence when $\beta>0$ to $\beta<0.$ Here we
formulate the approach in terms of the following \emph{zero-free hypothesis}
on the partition function $\mathcal{Z}_{N}(\beta)$ (defined in formula
\ref{eq:def of part function text II}): 
\begin{equation}
\mathcal{Z}_{N}(\beta)\neq0\,\text{on some \ensuremath{N-}independent neighborhood \ensuremath{\Omega} of \ensuremath{]-1,0]} in \ensuremath{\C}. }\label{eq:zero-free hyp}
\end{equation}
We also need to assume that $\mathcal{Z}_{N}(\beta)$ is finite on
a neighbourhood of $[-1,0]$ in $\R$ in a quantitative manner depending
on $N.$ This is made precise in the following result, which is a
refinement of \cite[Thm 7.9]{berm11}:
\begin{thm}
\label{thm:zero-free}Let $(X,\Delta)$ be a Fano orbifold. Assume
that there exists $\epsilon>0$ such that 
\begin{itemize}
\item $\mathcal{Z}_{N}(\beta)\leq C^{N}$ for $\beta=-(1+\epsilon)$ 
\item The zero-free hypothesis \ref{eq:zero-free hyp} holds 
\end{itemize}
Then $(X,\Delta)$ admits a Kähler-Einstein metric $\omega_{KE}$
and $\delta_{N}$ converge in law towards the normalized volume form
$dV_{KE}$ of $\omega_{KE}.$ More precisely, the convergence \ref{eq:conv of laws in theorem}
of laws holds and $-\frac{1}{N}\log\mathcal{Z}_{N}(\beta)$ converges
towards $F(\beta)$ in the $C^{\infty}-$topology on a neighborhood
of $]-1,0].$ Moreover, if $[-1,0]\Subset\Omega,$ then the convergence
holds on a neighborhood of $[-1,0].$ 
\end{thm}

\begin{proof}
First assume that $(X,\Delta)$ is a log Fano variety. Then the first
point in the theorem implies that $F$ admits a minimizer $\mu_{\beta}$
for any $\beta\in]-1-\epsilon,0[.$ Indeed, by the bound \ref{eq:upper bound on F N in pf}
$F(\beta)$ is bounded from below for any $\beta\in]-1-\epsilon,0].$
Thus, for any $\beta\in]-1-\epsilon,0[$ there exists $\delta>0$
such that $F_{\beta}\geq\delta E-\delta^{-1},$ which implies the
existence of $\mu_{\beta}$ (as recalled in the proof of Theorem \ref{thm:equiv cond for conv}).
In particular, taking $\beta=-1$ shows that $X$ admits a unique
Kähler-Einstein metric. Next, assume that $X$ is a Fano orbifold.
Then, the argument using the implicit function, employed in the proof
of Theorem \ref{thm:equiv cond for conv}, shows that after perhaps
replacing $\epsilon$ with a small positive number there exists a
unique solution $\varphi_{\beta}$ to the equation \ref{eq:complex ma eq with beta},
in the orbifold sense. In the case when $X$ is a Fano manifold it
was shown in the proof of \cite[Thm 7.9]{berm11} that $F(\beta)(=F(\mu_{\beta}))$
defines a real-analytic function on $]-(1+\epsilon),\infty[.$ Since
the proof only employs the implicit function theorem it applies more
generally when $(X,\Delta)$ is a Fano orbifold. Next, first consider
the case when $\mathcal{Z}_{N}(\beta)$ is zero-free on an $N-$independent
neighborhood $\Omega$ of $[-1,0]$ in $\C.$ By Theorem \ref{thm:1 implies 2 implies 3}
it will be enough to show that $\mathcal{Z}_{N}(\beta)^{1/N}\rightarrow e^{-F(\beta)}$
point-wise on $]-(1+\epsilon),\epsilon[.$ To this end first recall
that, by Theorem \ref{thm:beta pos}, the convergence holds when $\beta\geq0.$
Next, by the zero-free hypothesis $\mathcal{Z}_{N}(\beta)^{1/N}$
extends from $[-1,0]$ to a holomorphic function defined on a neighborhood
$\Omega$ of $[-1,0]$ in $\C.$ Moreover, by the first point 
\begin{equation}
\left|\mathcal{Z}_{N}(\beta)^{1/N}\right|\leq C\,\,\text{on \ensuremath{\Omega}}.\label{eq:estimate on Z N on Omega}
\end{equation}
(using that $\left|\mathcal{Z}_{N}(\beta)^{1/N}\right|\leq\mathcal{Z}_{N}(\Re\beta)^{1/N}\leq\mathcal{Z}_{N}(-1-\epsilon)^{1/N},$
which is uniformly bounded, by assumption). Hence, after perhaps passing
to a subsequence, we may assume that $\mathcal{Z}_{N_{j}}(\beta)^{1/N_{j}}$
converges uniformly in the $C^{\infty}-$topology on any compact subset
of $\Omega$ to a a holomorphic function $\mathcal{Z}(\beta),$ which,
in particular, defines a real-analytic function on $]-1-\epsilon,\epsilon[.$
But, when $\beta\geq0$ we have, as explained above, that $\mathcal{Z}(\beta)=e^{-F(\beta)}$
which extends to a real-analytic function on $]-1-\epsilon,\epsilon[.$
By the identity principle for real-analytic functions it thus follows
that $\mathcal{Z}_{N_{j}}(\beta)^{1/N_{j}}\rightarrow e^{-F(\beta)}$
for any $\beta$ in $]-1-\epsilon,\epsilon[,$ in the $C^{\infty}-$topology.
Since the limit is uniquely determined it thus follows that the whole
sequence $\mathcal{Z}_{N}(\beta)^{1/N}$ converges towards $e^{-F(\beta)},$
as desired. 

Finally, consider the case when it is only assumed that $\Omega$
is a neighborhood of $]-1,0]$ in $\C.$ By assumption, the sequence
of functions $F_{N}(\beta):=-\log(\mathcal{Z}_{N}(\beta)^{1/N})$
is uniformly bounded on $[-1-\epsilon,\epsilon].$ Since $F_{N}(\beta)$
is concave in $\beta$ it thus follows that $F_{N}(\beta)$ is uniformly
Lipschitz continuous on $[-1,0].$ Hence, by the Arzela-Ascoli theorem
we may, after perhaps passing to a subsequence, assume that $F_{N}(\beta)$
converges uniformly to continuous function $F_{\infty}(\beta)$ on
$[-1,0].$ By the previous argument $F_{\infty}(\beta)=F(\beta)$
on $]-1,0].$ But since $F_{\infty}$ and $F$ are both continuous
on $[-1,0]$ it follows that they also coincide at $\beta=-1,$ as
desired. 
\end{proof}
\begin{rem}
\label{rem:Lee-Yang}In statistical mechanical terms the $C^{\infty}-$convergence
of $N^{-1}\log\mathcal{Z}_{N}(\beta)$ amounts to the absence of phase
transitions \cite[Chapter 5]{ru}. It seems natural to expect that
the zero-free hypothesis \ref{eq:zero-free hyp} is satisfied as soon
as $X$ admits a Kähler-Einstein metric. Indeed, it can be viewed
as a strengthening of the real-analyticity of free energy $F(\beta)$
in some neighbourhood of $]0,1]$ in $\C$ (discussed in the proof
of the previous theorem). The zero-free hypothesis for general statistical
mechanical partition functions was introduced in the Lee-Yang theory
of phase transitions (and has been verified for some spin systems
and lattice gases \cite{y-l}). More precisely, originally Lee-Yang
considered zeros in the complexified field parameter $h$ called \emph{Lee-Yang
zeros}, while zeros with respect to the complexified inverse temperature
$\beta$ are called\emph{ Fisher zeros} \cite{f}. The role of $h$
in the present complex geometric setup is discussed in Remark \ref{rem:Lee-Yang with h}.
\end{rem}

As discussed in \cite[Section 6]{berm8 comma 5}, the bound in first
point in the previous theorem - which is independent of the choice
of metric $\left\Vert \cdot\right\Vert $ (up to changing the constant
$C)$ - can be viewed as an analytic (stronger) version of uniform
Gibbs stability (cf. \cite[Thm 6.7]{berm8 comma 5}). As shown in
\cite[Lemma 7.1]{berm11} the bound always holds for $\beta$ sufficiently
close to $0.$ More precisely,
\begin{equation}
\beta>-\text{lct }(-K_{X})\implies\mathcal{Z}_{N}(\beta)\leq C_{\beta}^{N}\label{eq:beta bigger than minus lct implies}
\end{equation}
for any $N(=N_{k})$, where $\text{lct }(L)$ denotes the global log
canonical threshold of a line bundle $L$ (whose definition is recalled
in the appendix). The proof exploits that $\text{lct }(-K_{X})$ coincides
with Tian's analytically defined $\alpha-$invariant $\alpha(-K_{X}).$
Accordingly, under the weaker hypothesis that $\mathcal{Z}_{N}(\beta)$
is zero-free, for $\beta$ in some $\epsilon-$neighborhood of $]-\text{lct }(X),0]$
in $\C,$ the convergence statements in the theorem hold when $\beta\in]-\text{lct }(X),0].$ 
\begin{rem}
If $\text{lct }(X)>1$ the first assumption in Theorem \ref{thm:zero-free}
is automatically satisfied. Such Fano orbifolds are called\emph{ exceptional}
(see \cite{c-p-s}, where two-dimensional exceptional hypersurfaces
in three-dimensional weighted projective space are classified). Exceptional
Fano orbifolds appear naturally in the Minimal Model Program as the
base of exceptional isolated affine singularities \cite{sh}. 
\end{rem}

\subsubsection{\label{subsec:The-strong-zero-free}The strong zero-free hypothesis}

The zero-free hypothesis is independent of the choice of basis in
$H^{0}(X,-kK_{X}).$ Indeed, under a change of basis $\det S^{(k)}$
gets multiplied by a non-zero scalar $c\in\C$ and hence $\mathcal{Z}_{N_{k}}(\beta)$
get multiplied by $c^{\beta/k}.$ However, it should be stressed that
the zero-free hypothesis depends, a priori, on the choice of metric
$\left\Vert \cdot\right\Vert .$ For example, there are reasons to
expect that it fails unless $\left\Vert \cdot\right\Vert $ has positive
curvature. Accordingly, the zero-free hypothesis might be more accessible
for special/canonical choices of positively curved metrics, such as
the Kähler-Einstein metric itself. This is illustrated by the following
example, where $\mathcal{Z}_{N_{k}}(\beta)$ can be explicitely computed:
\begin{example}
\label{exa:P n}When $X=\P_{\C}^{n}$ we have that $-K_{X}=\mathcal{O}(n+1)$
and hence the minimal value for $k$ is $k=1/(n+1),$ which means
that the minimal value for $N_{k}$ is $N_{k}=n+1.$ Taking $\left\Vert \cdot\right\Vert $
to be the Fubini-Study metric (which is Kähler-Einstein) the following
formula holds in the minimal case $N=n+1$ (where $c_{n}$ is a computable
positive constant), proved in the appendix (see Prop \ref{prop:formula in ex}):

\begin{equation}
\mathcal{Z}_{n+1}(\beta)=c_{n}\frac{\prod_{j=1}^{n}\Gamma\left(\beta(n+1)+j\right)}{\left(\Gamma\left(\beta(n+1)+n+1\right)\right)^{n}},\,\,\,\,\,\,\Gamma(a):=\int_{0}^{\infty}t^{a}e^{-t}\frac{dt}{t}\label{eq:form for Z N in ex}
\end{equation}
where $\Gamma(a)$ denotes the classical Gamma-function, which defines
a meromorphic function on $\C$ whose poles are located at $0,-1,-2,...$
(as follows from the functional relation $\Gamma(a+1)=a\Gamma(a)$).
Thus the first negative pole of $\mathcal{Z}_{N}(\beta)$ come from
the first pole of the factor corresponding to $j=1$ in the nominator
above, i.e. when $\beta=-1(n+1).$ Moreover, since $\Gamma(a)$ is
zero-free on all of $\C$ $\mathcal{Z}_{N}(\beta)$ is zero-free in
the maximal strip $\{\Re\beta>-1/(n+1)\}$ of holomorphicity (but
the meromorphic continuation $\mathcal{Z}_{N}(\beta)$ does have zeros
in $\C,$ coming from the poles of the denominator). 
\end{example}

In the light of this example it is tempting to speculate that the
following \emph{strong zero-free hypothesis} holds for Kähler-Einstein
metrics:
\[
\mathcal{Z}(\beta)\neq0,\,\,\,\text{when \ensuremath{\Re\beta>\text{\ensuremath{\max\{-}lct \ensuremath{(\mathcal{D}_{N}),-1\}}}}}.
\]
In other words, this means that $\mathcal{Z}_{N}(\beta)$ is zero-free
in the maximal strip inside $\left\{ \Re\beta>-1\right\} $ where
it is holomorphic. To provide some further evidence for the strong
zero-free property we note that if its holds, then the bound \ref{eq:beta bigger than minus lct implies}
combined with the proof of Theorem \ref{thm:zero-free} show that
for any given $\epsilon>0$ the function $F(\beta)$ on $]-\text{lct }(-K_{X})+\epsilon,\epsilon[\subset\R,$
induced by the Kähler-Einstein metric, is ``strongly real-analytic''
in the following sense: $F(\beta)$ extends to a bounded holomorphic
function on the infinity strip $]-\text{lct }(-K_{X})+\epsilon,\epsilon[+i\R\subset\C.$
This condition is much stronger than ordinary real-analyticity (which
only implies holomorphic extension to a finite strip). But it does
hold for the Kähler-Einstein metric. Indeed, in this case 
\[
F(\beta)\equiv0,\,\,\beta\in]-1,\infty[,
\]
 which trivially extends to a bounded holomorphic function on the
infinity strip. To prove the identity above first observe that when
$\omega_{0}=\omega_{KE}$ in the twisted Kähler-Einstein equation
\ref{eq:twisted KE with beta} is solved by $\omega_{\beta}=\omega_{KE}$
for \emph{any} $\beta$ (equivalently, in the case when $\omega_{0}=\omega_{KE},$
we have $\omega_{0}^{n}/V=dV_{(X,\Delta)}$ and hence the complex
Monge-Ampère equation \ref{eq:complex ma eq with beta} is solved
by $\varphi_{\beta}=0).$ But, as recalled above, for $\beta>-1$
the equation \ref{eq:complex ma eq with beta} admits a unique solution
and hence 
\[
F(\beta)=F_{\beta}(dV_{KE})=0
\]
(using the vanishing \ref{eq:E vanishes on reference vol} combined
with the vanishing $\text{Ent}(\mu)=0$ when $\mu=dV_{KE}=dV_{\Delta}).$
In fact this argument shows that $F(\beta)\equiv0$ on all of $[-1,\infty[.$
Moreover, if $\text{Aut \ensuremath{(X)_{0}}}$ is trivial then there
exists an $\epsilon>0$ such that $F(\beta)\equiv0$ on all of $]-1-\epsilon,\infty[,$
as follows form the argument using the implicit function theorem,
employed in the proof of Theorem \ref{thm:equiv cond for conv}. This
argument suggests that when $\text{Aut \ensuremath{(X)_{0}}}$ is
trivial one can, perhaps, expect the strong zero-free property to
even hold in the larger region where $\Re\beta>\max\{-\text{lct \ensuremath{(\mathcal{D}_{N}),-1-\epsilon\}}}$
for some $\epsilon>0.$
\begin{rem}
Coming back to example \ref{exa:P n} it is natural to ask if there
exists an explicit formula for $\mathcal{Z}_{N}(\beta)$ when $X=\P_{\C}^{n}$
for general $N,$ generalizing formula \ref{eq:form for Z N in ex}
(or, more precisely, for any $N$ of the form $N=N_{k})$ ? However,
as discussed in Remark \ref{rem:bernstein} this problem appears to
be open even when $n=1.$ But one interesting consequence of formula
\ref{eq:form for Z N in ex} is that it reveals that that in the case
when $X=\P_{\C}^{n}$ and $N$ is minimal
\[
\text{lct}(\mathcal{D}_{N})=\text{lct}(-K_{X}).
\]
 (since $\text{lct}(-K_{X})=1/(n+1).$ This shows that the estimate
in formula \ref{eq:beta bigger than minus lct implies} is sharp (in
the sense that there are cases where it fails for $\beta\leq-\text{lct }(-K_{X})).$
The point of Conjecture \ref{conj:Fano with triv autom intr}, however,
is that it only requires that $\text{lct}(\mathcal{D}_{N_{k}})>1$
when $N_{k}$ is sufficiently large. Similarly, in the case of $\P_{\C}^{n},$
where $\text{Aut \ensuremath{(X)_{0}}}\neq\{I\},$ the corresponding
conjecture only requires that $\text{lct}(\mathcal{D}_{N_{k}})\rightarrow1,$
when $N_{k}\rightarrow\infty$ (see \cite[Conj 3.8]{berm11}). For
example, when $X=\P_{\C}^{1}$ one has $\text{lct}(\mathcal{D}_{N})=(N-1)/N$
(by Theorem \ref{thm:-P one with arb w}) which indeed tends to $1$
as $N\rightarrow\infty$ (and equals $1/2$ when $N=2,$ which is
the minimal case). 
\end{rem}

\subsubsection{Allowing singular metrics $\left\Vert \cdot\right\Vert $}

Alternatively, when $X$ is a Fano manifold, one can take $\left\Vert \cdot\right\Vert $
to be the singular metric induced by the anti-canonical $\Q-$divisor
$\Delta_{m}$ defined by the zero-locus of a holomorphic section of
$-mK_{X},$ assuming that $m>0$ and the zero-locus is non-singular
(which ensures that the corresponding singular volume form $dV$ has
a density in $L_{loc}^{p}$ for some $p>1).$ In other words, the
curvature of $\left\Vert \cdot\right\Vert $ is given by the positive
current $[\Delta_{m}]$ supported on $\Delta_{m}.$ Then Theorem \ref{thm:zero-free}
still applies. Indeed, in the proof one can apply the implicit function
to the wedge-Hölder spaces appearing in \cite{do2,jmr}, which are
independent of $\beta$ (see, in particular, \cite[Cor 3.5]{jmr}).
In this singular setup the corresponding equations \ref{eq:twisted KE with beta}
become Donaldson's variant of Aubin's continuity equations
\begin{equation}
\mbox{\ensuremath{\mbox{Ric}}\,\ensuremath{\omega_{\beta}}}=t\omega_{\beta}+(1-t)[\Delta_{m}],\,\,\,t=-\beta\label{eq:Dons cont eq}
\end{equation}
that were used in the proof of the YTD conjecture in \cite{c-d-s},
by deforming $t$ from an initial small value, where there always
exists a solution (by \cite[Thm 1.5]{berm6}) to $t=1,$ assuming
that $X$ is K-stable. In other words, $\beta$ is deformed down to
$-1.$ In the present probabilistic approach the (potential) advantage
of employing the singular metric on $-K_{X}$ induced by the $\Q-$divisor
$\Delta_{m}$ is that the corresponding partition function $\mathcal{Z}_{N}(\beta)$
is encoded by purely algebraic data; the divisors $\mathcal{D}_{N}$
and $\Delta_{m}$ on $X^{N}$ and $X,$ respectively. In this case
combining \cite[Prop 6.2]{berm6} with \cite[Lemma 7.1]{berm11} gives
\[
\beta>-\min\left\{ \text{lct }(-K_{X}),\text{lct }(-K_{X|\Delta_{m}})\right\} \implies\mathcal{Z}_{N}(\beta)\leq C_{\beta}^{N},
\]
 where $-K_{X|\Delta_{m}}$ denotes the restriction of $-K_{X}$ to
the support of $\Delta_{m}.$ More generally, it seems natural to
expect that Theorem \ref{thm:zero-free} holds for \emph{any }log
Fano variety $(X,\Delta)$ (when $\left\Vert \cdot\right\Vert $ is
either a smooth metric on $K_{X}+\Delta$ with positive curvature
or the singular metric defined by\emph{ any} klt $\Q-$divisor in
$-(K_{X}+\Delta)).$ In the case when $\Delta+\Delta_{m}$ defines
a divisor whose components are non-singular and mutually non-intersecting
the aforementioned results in \cite{do2,jmr} still apply.

\subsubsection{\label{subsec:Deforming-the-divisor}Deforming the divisor $\Delta$}

Sometimes it is advantageous to keep $\beta=-1$ and instead deform
the divisor $\Delta$ as follows. Given a log Fano variety $(X,\Delta)$
and a positive real number $k$ such that $-k(K_{X}+\Delta)$ is well-defined
line bundle $\mathcal{L},$ i.e. defines an element in the integral
lattice $H^{2}(X,\Z)$ of $H^{2}(X,\R),$ consider the affine subspace
$\mathcal{A}$ of $\R^{M+1}$ of all $(\boldsymbol{w},s)$ which are
``admissible'' in the sense that 
\begin{equation}
-(K_{X}+\Delta(\boldsymbol{w}))=s\mathcal{L},\label{eq:cond on w and s}
\end{equation}
 where $\Delta(\boldsymbol{w})$ denotes the divisor with the same
$M$ irreducible components as the given divisor $\Delta$ and coefficients
$\boldsymbol{w}\in\R^{M}.$ In particular, $(\boldsymbol{w}_{0},k^{-1})$
is ``admissible'', where $\boldsymbol{w}_{0}\in\R^{M}$ denotes
the coefficients of the initial divisor $\Delta.$ If there exists
$(\boldsymbol{w}_{1},s_{1})\in\mathcal{A}$ such that $K_{X}+\Delta(\boldsymbol{w}_{1})>0$
(and hence $s_{1}<0)$ the conclusion of Theorem \ref{thm:zero-free}
still applies if the corresponding function partition function $\mathcal{Z}_{N},$
viewed as a meromorphic function on $\C^{M+1},$ satisfies 
\begin{itemize}
\item $\mathcal{Z}_{N}\leq C_{0}^{N}$ in a neighborhood in $\R^{M+1}$
of $(\boldsymbol{w}_{0},k^{-1})$ 
\item $\mathcal{Z}_{N}\neq0$ in an $N-$independent neighborhood of the
line-segment in $\C^{M+1}$ connecting $(\boldsymbol{w}_{0},k^{-1})$
and $(\boldsymbol{w}_{1},s_{1}).$
\end{itemize}
More precisely, as discussed in the previous section, in order to
apply the implicit function theorem in Banach spaces the appropriate
linear PDE-theory needs to be in place. For example, by \cite{do2,jmr}
this is the case when the components of $\Delta$ are non-singular
and mutually non-intersecting (results concerning the case when $(X,\Delta)$
is log smooth are announced in \cite{m-r}). The previous proof can
then by applied to the meromorphic function $\mathcal{Z}_{N}(t)$
on $\C$ defined by the partition functions associated to the line-segment
$I\Subset\C^{m+1}$ connecting the initial $(\boldsymbol{w}_{0},k^{-1})$
with $(\boldsymbol{w}_{1},s_{1})$ (where $t$ denotes the complexification
of the standard parametrization of $I).$ In this situation the estimate
\ref{eq:estimate on Z N on Omega} still holds, i.e. $\left|\mathcal{Z}_{N}(t)^{1/N}\right|\leq C$
on some $N-$independent neighborhood $\Omega$ of $[0,1]$ in $\C.$
Indeed, by assumption, the estimate holds with constant $C_{0}$ in
a neighborhood of $t=0$ and, moreover, it trivially holds with a
constant $C_{1}$ when $t$ is close to $t=1.$ Since $\log\mathcal{Z}_{N}(t)$
is convex wrt $t\in[0,1]$ one can thus take $C=\max\{C_{0},C_{1}\}.$ 

\section{\label{sec:Intermezzo:-a-zero-free}Intermezzo: a zero-free hypothesis
for polarized manifolds $(X,L)$ and the Calabi-Yau equation}

Before turning to the case of log Fano curves, we make a digression
on general polarized manifolds $(X,L),$ i.e. a compact complex manifold
$X$ endowed with an ample line bundle $L.$ To a metric $\left\Vert \cdot\right\Vert $
on $L$ and a volume form $dV$ on $X$ we may attach partition functions
$\mathcal{Z}_{N}(\beta),$ by replacing the log canonical line bundle
$-(K_{X}+\Delta)$ with $L$ and $dV_{\Delta}$ with $dV$ in formula
\ref{eq:def of part function text II}:
\begin{equation}
\mathcal{Z}_{N}(\beta):=\int_{X^{N}}\left\Vert \det S^{(k)}\right\Vert ^{2\beta/k}dV^{\otimes N},\label{eq:def of Z N for L}
\end{equation}
 where $k$ is a given positive integer and $N$ denotes the dimension
of $H^{0}(X,kL).$ This is the general setup considered in \cite{berm8},
where the corresponding free energy functional is of the form
\[
F_{\beta}(\mu):=\beta E(\mu)+\text{End (\ensuremath{\mu)},}
\]
 where $E(\mu)$ denotes the pluricomplex energy of $\mu$ with respect
to the normalized curvature form $\omega$ of the metric $\left\Vert \cdot\right\Vert $on
$L$ and $\text{Ent}(\mu)$ denotes the entropy of $\mu$ relative
to $dV.$ The minimizers $\mu_{\beta}$ of $F_{\beta}(\mu)$ are of
the form
\[
\mu_{\beta}=e^{\beta\varphi_{\beta}}dV
\]
 for a smooth solution $\varphi_{\beta}$ of the complex Monge-Ampère
equation 
\begin{equation}
\frac{1}{V}(\omega+\frac{i}{2\pi}\partial\bar{\partial}\varphi_{\beta})^{n}=e^{\beta\varphi_{\beta}}dV.\label{eq:MA eq with beta dV}
\end{equation}

\begin{rem}
In the case when $\beta=k$ and $X$ is a Riemann surface the corresponding
partition function $\mathcal{Z}_{N}(\beta)$ coincides with the $L^{2}-$norm
of the Laughlin wave function for the (integer) Quantum Hall state
on $X,$ subject to the magnetic two-form $ik\omega$ \cite{k-m-m-w}.
Accordingly, as shown in \cite{berm 1 komma 5}, in this case (and
for any dimension of $X$) the corresponding large $N-$limit is described
by the minimizers $F_{\beta}(\mu)/\beta,$ as $\beta\rightarrow\infty,$
i.e. of $E(\mu).$ However, here we are concerned with the case when
$\beta$ is fixed, where entropy enters the picture and dominates
when $\beta$ is close to $0.$
\end{rem}

Consider, in this general setup, the following \emph{weak zero-free
hypothesis}: 
\begin{equation}
\mathcal{Z}_{N}(\beta)\neq0\,\text{on some \ensuremath{N-}independent neighborhood \ensuremath{\Omega} of \ensuremath{0} in \ensuremath{\C}. }\label{eq:zero-free hyp-1}
\end{equation}
It implies a weaker form of the upper bound hypothesis \ref{eq:upper bound property}
on the mean energy: 
\begin{thm}
\label{thm:mean energy for L}Let $(X,L)$ be a polarized manifold.
Given a a metric $\left\Vert \cdot\right\Vert $on $L$ and a volume
form $dV$ on $X,$ assume that the corresponding partition functions
$\mathcal{Z}_{N}(\beta)$ satisfy the weak zero-free hypothesis above.
Then $-\frac{1}{N}\log\mathcal{Z}_{N}(\beta)$ converges towards $F(\beta)$
in the $C^{\infty}-$topology on a neighborhood of $0$ in $\R.$
In particular, the mean energy of $dV^{\otimes N}$ converges towards
the pluricomplex energy $E(dV)$ of $dV:$ 
\begin{equation}
\lim_{N\rightarrow\infty}\int_{X^{N}}E^{(N)}dV^{\otimes N}=E(dV),\,\,\,\,\,\,\,\,\,E^{(N)}(x_{1},...,x_{N}):=-\frac{1}{kN}\log\left(\left\Vert \det S^{(k)}(x_{1},...,x_{N})\right\Vert ^{2}\right)\label{eq:conv of mean energy of dV}
\end{equation}
\end{thm}

\begin{proof}
In general, given a metric $\left\Vert \cdot\right\Vert $on $L$
and a volume form $dV$ on $X$ there exists $\epsilon>0$ such that
$F(\beta)$ is real-analytic on $]-\epsilon,\epsilon[.$ Indeed, this
follows, as before, from an application of the implicit function theorem
at $\beta=0.$ Moreover, by the argument discussed in connection to
formula \ref{eq:beta bigger than minus lct implies},
\begin{equation}
\beta>-\text{lct }(L)\implies\mathcal{Z}_{N}(\beta)\leq C_{\beta}^{N}\label{eq:beta bigger than minus lct implies-1}
\end{equation}
In particular, the estimate holds when $\beta>-\epsilon$ for $\epsilon$
sufficiently small. The $C^{\infty}-$convergence of $-\frac{1}{N}\log\mathcal{Z}_{N}(\beta)$
towards $F(\beta)$ then follows exactly as in the proof of Theorem
\ref{thm:zero-free}. Finally, the convergence of the first derivatives
at $\beta=0$ yields the convergence \ref{eq:conv of mean energy of dV}. 
\end{proof}
We next show that a variant of the weak zero-free hypothesis yields
canonical approximations $\varphi_{N}$ of the solution of the \emph{Calabi-Yau
equation,} i.e. the equation obtained by setting $\beta=0$ in equation
\ref{eq:MA eq with beta dV}:
\begin{equation}
\frac{1}{V}(\omega+\frac{i}{2\pi}\partial\bar{\partial}\varphi)^{n}=dV\label{eq:CY eq for dV}
\end{equation}
for a smooth function $\varphi$ on $X.$ By Yau's theorem \cite{y}
there exists a unique smooth solution $\varphi$ with vanishing average
on $(X,dV).$ Given a volume form $dV$ with unit total volume the
canonical approximation $\varphi_{N}$ in question is defined by the
following integral formula:
\begin{equation}
\varphi_{N}(x):=\int\frac{1}{k}\log\left(\left\Vert \det S^{(k)}(x,x_{2},...,x_{N})\right\Vert ^{2}\right)dV^{\otimes N-1}-c_{N},\label{eq:def of phi N}
\end{equation}
where $c_{N}$ is the constant ensuring that the average of $\varphi_{N}$
on $(X,dV)$ vanishes: 
\[
c_{N}:=\int_{X^{N}}\frac{1}{k}\log\left(\left\Vert \det S^{(k)}(x,x_{2},...,x_{N})\right\Vert ^{2}\right)dV^{\otimes N}.
\]
For a given smooth function $u$ on $X$ denote by $\mathcal{Z}_{N}(\beta,h)$
the function on $\R^{2}$ obtained by replacing $dV$ in formula \ref{eq:def of Z N for L}
with $e^{hu}dV:$ 
\begin{equation}
\mathcal{Z}_{N}(\beta,h):=\int_{X^{N}}\left\Vert \det S^{(k)}\right\Vert ^{2\beta/k}(e^{hu}dV)^{\otimes N}.\label{eq:def of Z N beta h}
\end{equation}
 
\begin{thm}
\label{thm:calabi-yau}Let $(X,L)$ be a polarized manifold and $\left\Vert \cdot\right\Vert $a
metric on $L.$ Given a volume form $dV$ on $X$ with unit total
volume assume that 
\begin{equation}
\mathcal{Z}_{N}(\beta,h)\neq0\,\text{on some \ensuremath{N-}independent neighborhood \ensuremath{\Omega} of \ensuremath{(0,0)} in \ensuremath{\C^{2}}. }\label{eq:zero-free hypothesis cy}
\end{equation}
 for any smooth function $u$ on $X$ (where $\Omega$ depends on
$u).$ Then the functions $\varphi_{N},$ defined by formula \ref{eq:def of phi N},
converge in $L^{1}(X)$, as $N\rightarrow\infty,$ to the unique smooth
solution $\varphi$ of the Calabi-Yau equation \ref{eq:CY eq for dV}
satisfying $\int_{X}\varphi dV=0.$ 
\end{thm}

\begin{proof}
First observe that $\varphi_{N}(x)$ is $\omega-$psh, since it is
a superposition of the $\omega-$psh functions $\log\left(\left\Vert \det S^{(k)}(x,x_{2},...,x_{N})\right\Vert ^{2}\right).$
Hence, by standard properties of $\omega-$psh functions the $L^{1}-$convergence
in question is equivalent to weak convergence. In other words, it
is equivalent to proving that for any given smooth function $u\in C^{\infty}(X)$
\[
\lim_{N\rightarrow\infty}\int\varphi_{N}udV=\int\varphi dV.
\]
 Moreover, since the integrals on both side of the previous equality
vanish for $u=1$ it is enough to prove the convergence for any $u\in C^{\infty}(X)$
satisfying $\int udV=0.$ To this end fix such a function $u$ and
consider the corresponding partition functions $\mathcal{Z}_{N}(\beta,h),$
defined by formula \ref{eq:def of Z N beta h}. A direct calculation
reveals that

\begin{equation}
\int\varphi_{N}udV=\frac{\partial}{\partial h}\frac{\partial}{\partial\beta}N^{-1}\log\mathcal{Z}_{N}(\beta,h),\,\,\,\text{at \ensuremath{(\beta,h)=(0,0).} }\label{eq:phi N u as derivative}
\end{equation}
By assumption there exists a neighborhood $\Omega$ of $(0,0)$ in
$\C^{2}$ where $\log\mathcal{Z}_{N}(\beta,h)$ is holomorphic. Moreover,
by Theorem \ref{thm:mean energy for L} 
\[
-N^{-1}\log\mathcal{Z}_{N}(\beta,h)\rightarrow F(\beta,h):=\inf_{\mathcal{\mu\in P}(X)}\left(\beta E(\mu)-h\int_{X}udV+\text{Ent}(\mu)\right),
\]
 in the $C_{loc}^{\infty}-$topology on $\Omega,$ where $\text{Ent}(\mu)$
denotes the entropy of $\mu$ relative to $dV.$ In particular, the
convergence of the second derivatives at $(0,0)$ yields, by formula
\ref{eq:phi N u as derivative}, 
\[
\lim_{N\rightarrow\infty}\int\varphi_{N}udV=-\frac{\partial}{\partial h}\frac{\partial F(\beta,h)}{\partial\beta}\,\,\,\text{at}\,\ensuremath{(\beta,h)=(0,0).}
\]
Since $dV$ is the unique minimizer of $F_{\beta}$ when $\beta=0,$
\[
\frac{\partial F(\beta,h)}{\partial\beta}=E(dV_{h}),\,\,\,\,dV_{h}:=dVe^{hu}/\int_{X}dVe^{hu}.
\]
 The proof is thus concluded by invoking the property \ref{eq:first var of E}
of the functional $E,$which gives 
\[
-\frac{E(dV_{h})}{\partial h}_{|u=0}=\int_{X}\varphi udV.
\]
\end{proof}
In the particular case when $X$ is a Calabi-Yau manifold - i.e. when
some power of $K_{X}$ is trivial - we can apply the previous theorem
to the canonical normalized volume form $dV$ on $X,$ 
\[
dV:=\frac{\left(s_{m}\wedge\overline{s}_{m}\right)^{1/m}}{\int_{X}\left(s_{m}\wedge\overline{s}_{m}\right)^{1/m}},
\]
 where $s_{m}$ trivializes $mK_{X}$ for some positive integer $m.$
Then the corresponding convergence implies that the positive $(1,1)-$currents
\[
\omega_{N}:=\frac{i}{2\pi k}\int\partial\bar{\partial}\log\left(\left|\det S^{(k)}(\cdot,x_{2},...,x_{N})\right|^{2}\right)dV^{\otimes N-1}
\]
 converge weakly towards the unique Calabi-Yau metric $\omega_{CY}$
on $X$ in $c_{1}(L),$ i.e. towards the unique Ricci flat Kähler
metric in $c_{1}(L).$ Note that, by the Poincaré-Lelong formula,
$\omega_{N}$ is the average over $X^{N-1}$ of the currents of integration
defined by the zero-loci in $X$ of the holomorphic sections $\det S^{(k)}(\cdot,x_{2},...,x_{N}).$ 
\begin{rem}
\label{rem:Lee-Yang with h}It seems natural to expect that the zero-free
hypothesis \ref{eq:zero-free hypothesis cy} is always satisfied.
Indeed, it can be viewed as a strengthening of the real-analyticity
of the free energy $F(\beta,h)$ in some neighborhood of $(0,0)$
in $\C^{2}$(discussed in the proof of the previous theorem). This
expectation is in line with corresponding expectations in the Lee-Yang
theory of phase transitions \cite{y-l}, where the the role of $\beta$
and $h/\beta$ is played by the inverse temperature and the field
strength, respectively (see the discussion in the introduction of
\cite{k-f}).
\end{rem}

When $X$ is a compact complex curve, i.e. $n=1,$ the convergence
in Theorem \ref{thm:mean energy for L} and Theorem \ref{thm:calabi-yau}
can, unconditionally, be deduced from the bosonization formula for
$\det S^{(k)}(x_{1},...,x_{N}).$ \cite{a-g-b--} To the leading order
this formula expresses $\left\Vert \det S^{(k)}(x,x_{2},...,x_{N})\right\Vert $
as a product of $G(x_{i},x_{j}),$ where $G$ is the Green's function
for the Laplacian $i\partial\bar{\partial}$ (see Lemma \ref{lem:standard id}
for the case when $X=\P_{\C}^{1}).$ 

\section{\label{sec:A-case-study:}The case of log Fano curves }

Let $X$ be  the complex projective line $\P_{\C}^{1}.$ Fix an $\R-$divisor
$\Delta$ on $X,$ i.e.
\[
\Delta:=\sum_{1=1}^{m}p_{i}w_{i}
\]
 for given points $p_{1},...,p_{m}$ on $X$ and with real coefficients/weights
$w_{i}$ and assume that 
\[
w_{i}<1.
\]
 In contrast to Section \ref{subsec:Main-new-results} we thus allow
$w_{i}$ to be negative. Assume that $(X,\Delta)$ is a log Fano manifold,
i.e. the anti-canonical line bundle of $(X,\Delta)$ is positive:

\[
L:=-(K_{X}+\Delta)>0
\]
Since $X$ is a complex curve the assumption that $L$ is positive
simply means that its degree $d_{L}$ is positive:

\begin{equation}
d_{L}=2-\sum w_{i}>0\label{eq:def of d L}
\end{equation}
Given a positive real number $k$ such that and assume that $kL$
defines a line bundle, i.e. $kd_{L}$ is an integer set 
\[
N_{k}:=\dim H^{0}(X,kL).
\]
To the log Fano curve $(X,\Delta)$ we attach (as in the beginning
of Section \ref{sec:Conditional-convergence-results}) the following
symmetric probability measure on $X^{N_{k}},$ 
\[
\mu_{\Delta}^{(N_{k})}=\frac{1}{\mathcal{Z}_{N_{k}}}\left|\det S^{(k)}(z_{1},...,z_{N})\right|^{-2/k}|s_{\Delta}|^{-2}(z_{1})\cdots|s_{\Delta}|^{-2}(z_{N_{k}}),
\]
which is well-defined precisely when $\mathcal{Z}_{N_{k}}<\infty.$
The following result implies Theorem \ref{thm:log curve intro} (concerning
the case when $w_{i}>0)$: 
\begin{thm}
\label{thm:log Fano curve text}Let $(X,\Delta)$ be a log Fano curve.
Then the following is equivalent:
\begin{itemize}
\item $\mathcal{Z}_{N_{k}}<\infty$ for $k$ sufficiently large 
\item The following weight condition holds: 
\begin{equation}
w_{i}<\sum_{i\neq j}w_{j},\,\,\,\forall i\label{eq:weight condition}
\end{equation}
 
\end{itemize}
Moreover, if any of the conditions above hold then the law of the
empirical measure $\delta_{N}$ on $(X^{N_{k}},\mu_{\Delta}^{(N_{k})})$
satisfies a LDP with speed $N$ and rate functional $F_{-1}-\inf_{\mathcal{P}(X)}F_{-1}$
(where $F_{-1}$ is the free energy functional on $\mathcal{P}(X)$
defined in Section \ref{subsec:The-case beta neg}, which coincides
with the Mabuchi functional for $(X,\Delta)$). 

\end{thm}

\begin{rem}
In particular, if the weight condition above holds then $F_{-1}$
is lsc on $\mathcal{P}(X)$ (since, in general, any rate functional
for a LDP is lsc) and thus admits a minimizer. The existence of a
minimizer was first shown in \cite{tr} using a different variational
argument. By the general results for log Fano varieties $(X,\Delta)$
in \cite{bbegz} any minimizer satisfies the Kähler-Einstein equation
for $(X,\Delta).$ In general, a solution is not uniquely determined
(see \cite[Remark 2]{l-t}). However, when $w_{i}>0$ the uniqueness
in the case of the Riemann sphere was shown in \cite{l-t} (see \cite{c-d-s,bbegz}
for the general higher dimensional log Fano case). 
\end{rem}

To prove the previous theorem we first recall some standard identifications
(see \cite[Section 3.7]{berm10b}). Fixing a point $p_{\infty}$ we
identify $X-\{p_{\infty}\}$ with $\C.$ The point $p_{\infty}$ induces
a trivialization $e_{\infty}$ of the restriction of the hyperplane
line bundle $\mathcal{O}(1)\rightarrow\P_{\C}^{1}$ to $\C$ (vanishing
at $p_{\infty})$ and thus the space $H^{0}(X,d\mathcal{O}(1))$ of
all global holomorphic sections of the $d$th tensor power of the
hyperplane line bundle $\mathcal{O}(1)\rightarrow X$ may be identified
with the space of all polynomials in $z$ of degree at most $d.$
Moreover, the anti-canonical line bundle $-K_{X}$ of $X$ may be
identified with $2\mathcal{O}(1)$ and $s_{\Delta}$ with a (multivalued)
holomorphic section of $\sum w_{i}\mathcal{O}(1).$ In particular,
we identify 
\[
kL\longleftrightarrow kd_{L}\mathcal{O}(1)=k\left(2-\sum_{i=1}^{m}w_{i}\right)\mathcal{O}(1),
\]
 (recall that we are assuming that $kd_{L}$is an integer). Thus $H^{0}(X,kL)$
gets identified with the space of all polynomials in $z$ of degree
at most $k\left(2-\sum_{i=1}^{m}w_{i}\right).$ This identification
reveals that 
\begin{equation}
N_{k}=kd_{L}+1\label{eq:N k in terms of d L}
\end{equation}
Fix the standard basis of monomials $1,z,z^{2},...$ in $H^{0}(X,kL).$
Then the corresponding section $\det S^{(k)}$ over $X^{N_{k}}$ gets
identified with the usual Vandermonde determinant on $\C^{N_{k}}$:
\begin{equation}
\det S^{(k)}\longleftrightarrow D(z_{1},...,z_{N_{k}}):=\det_{i,j\leq N_{k}}(z_{i}^{j})\label{eq:def of D Vanderm}
\end{equation}
Next, we identify $X$ with the unit-sphere $S^{2}$ in $\R^{3},$
using the standard stereographic projection, so that the fixed point
$p_{\infty}\in X$ corresponds to the ``north-pole'' $(0,0,1)$
in $S^{2}:$

\[
z\mapsto x:=\left(\frac{z+\bar{z}}{1+|z|^{2}},\frac{z-\bar{z}}{1+|z|^{2}},\frac{-1+|z|^{2}}{1+|z|^{2}}\right),\,\,\,\C\rightarrow\R^{3}
\]
Denote by $dV_{X}$ the area form of the standard round metric on
$S^{2}$ and by $G$ the following lsc function on $X$ 
\[
G(x,y):=-\log\left\Vert x-y\right\Vert ,
\]
expressed in terms of the Euclidean norm on $\R^{3}.$
\begin{lem}
\label{lem:standard id}In terms of the standard identifications over
$\C$ 
\[
\left|\det S^{(k)}(z_{1},...,z_{N})\right|^{-2/k}|s_{\Delta}|^{-2}(z_{1})\cdots|s_{\Delta}|^{-2}(z_{N_{k}})=\frac{1}{\left(\prod_{i\neq j}\left|z_{i}-z_{j}\right|\right)^{\frac{d_{L}}{N-1}}}\frac{1}{\prod_{i}}\frac{1}{\left|z_{i}-p_{j}\right|^{2w_{j}}}
\]
(where $d_{L}$ is defined in formula \ref{eq:def of d L}). As a
consequence, on $X:=\P_{\C}^{1}$ the probability measure $\mu_{\Delta}^{(N)}$
may be expressed as 
\begin{equation}
\mu_{\Delta}^{(N)}=\frac{1}{Z_{N}}e^{\frac{d_{L}}{N-1}\sum_{i\neq j\leq N}G(x_{i},x_{j})}dV^{\otimes N},\,\,\,dV:=e^{\sum_{i\leq m}w_{i}G(x,p_{i})}dV_{X}\label{eq:form mu N delta in lemma}
\end{equation}
\end{lem}

\begin{proof}
First, factorizing the Vandermonde determinant $D(z_{1},...,z_{N_{k}})$
on $\C^{N}$ reveals that $D(z_{1},...,z_{N_{k}})$ is the product
of $(z_{i}-z_{j})$ over all $i,j$ in $\{1,...,N\}$ such that $i<j.$
Hence, 

\begin{equation}
\left|D(z_{1},...,z_{N_{k}})\right|^{2}=\prod_{i\neq j}\left|z_{i}-z_{j}\right|\label{eq:D N factorized}
\end{equation}
Since $N_{k}=kd_{L}+1,$ we have that $k=(N-1)/d_{L}$ and hence the
first formula of the lemma follows. To prove the second one first
recall that in the general setting of log Fano manifolds $(X,\Delta)$
the measure $\mu_{\Delta}^{(N)}$ may be expressed as in formula \ref{eq:mu N Delta in terms of metric}.
In the present case we take $\left\Vert \cdot\right\Vert $ to be
the metric on $L$ induced from the Fubini-Study metric $\left\Vert \cdot\right\Vert _{FS}$
on $\mathcal{O}(1)$ under the identification of $L$ with $d_{L}\mathcal{O}(1).$
Recall that
\[
\left\Vert e_{\infty}\right\Vert _{FS}^{2}=e^{-\phi_{FS}(z)},\,\,\,\phi_{FS}(z):=\log(1+|z|^{2})
\]
Hence formula \ref{eq:form mu N delta in lemma} follows from the
following two facts. First, 
\begin{equation}
\left\Vert z-w\right\Vert _{FS}^{2}:=\left|z-w\right|^{2}e^{-\phi_{FS}(z)}e^{-\phi_{FS}(w)}\label{eq:FS}
\end{equation}
 is proportional to the squared norm in $\R^{3}$ under stereographic
projection and secondly

\[
\left\Vert dz\right\Vert _{FS}^{2}:=\left\Vert e_{\infty}^{\otimes2}\right\Vert _{FS}^{2}:=e^{-2\phi_{FS}}
\]
is proportional to the density of $dV_{X}.$ These are well-known
relations that can be checked explicitly, but they also follow readily
from their invariance under the isometry group of $S^{2}.$
\end{proof}
Next, we recall the following general LDP \cite[Thm 1.5]{berm10},
generalizing the convergence in probability established in \cite{clmp,k}
for the point-vortex model in a planar compact domain. Given a symmetric
function $W$ on a compact metric space $X,$ a measure $\mu_{0}$
on $X$ and $p\in\R$ set 

\[
\mu^{(N)}[p]=\frac{1}{Z_{N[p]}}e^{-p\frac{1}{N}\sum_{x_{i}\neq x_{j}}W(x_{i},x_{j})}\mu_{0}^{\otimes N},\,\,\,Z_{N}[p]:=\int_{X^{N}}e^{-p\frac{1}{N}\sum_{x_{i}\neq x_{j}}W(x_{i},x_{j})}\mu_{0}^{\otimes N},
\]
 assuming that $Z_{N}[p]<\infty.$ 
\begin{thm}
\label{thm:LDP two-point interaction}Let $X$ be a compact metric
space, $\mu_{0}$ a measure on $X$ and $W$ a lower semi-continuous
symmetric measurable function on $X^{2}$ and $p_{0}$ a negative
number such that
\begin{equation}
\sup_{x\in X}\int_{X}e^{-p_{0}W(x,y)}\mu_{0}(y)<\infty\label{eq:integr condition}
\end{equation}
Then, for any $p>p_{0}$ the normalizing constant $Z_{N}[p]$ is finite
and the law of the empirical measure $\delta_{N}$ on $(X^{N},\mu^{(N)}[p])$
satisfies a LDP with a rate functional 
\[
F_{p}-\inf_{\mathcal{P}(X)}F_{p},\,\,\,\,F_{p}(\mu):=p\int_{X\times X}W\mu\otimes\mu+\text{Ent}_{\mu_{0}}(\mu)
\]
\end{thm}

\begin{proof}
It may be illuminating to reformulate the proof given in \cite{berm10}
in terms of the conditional convergence result in Theorem \ref{thm:1 implies 2 implies 3}.
First, the finiteness of $Z_{N}[p]$ follows readily from the arithmetic-geometric
means inequality, using the integrability condition \ref{eq:integr condition}.
A refinement of this argument also yields a priori estimates on each
$j-$point correlation measure on $X^{j},$ building on \cite[Section 3.2.4]{berm11},
showing that its density is uniformly bounded in $L^{p}(\mu_{0}^{\otimes j})$
for any $p>1.$ Applying this estimate to $j\leq2$ shows that the
``upper bound hypothesis'' \ref{eq:upper bound property} of the
energy is satisfied. A twist of this argument also yields the stronger
form of the upper bound hypothesis with respect to any given continuous
function $\Phi(\mu),$ as formulated in Theorem \ref{thm:1 implies 2 implies 3},
and thus also the LDP.
\end{proof}
In the present case we thus have 
\[
W(z,w)=-d_{L}\log\left\Vert z-w\right\Vert ,\,\,\,p=\beta\frac{N-1}{N}
\]
Moreover, 
\begin{equation}
\int_{X}W\mu\otimes\mu=E(\mu)+C\label{eq:energy on P one is integral W}
\end{equation}
 for some constant $C.$ Indeed, by a simple scaling argument it is
enough to consider the case when $d_{L}=1.$ Then we can write $W(x,y)=G(x,y)/2,$
where $G(x,y)=-\log(\left\Vert z-w\right\Vert ^{2})$ has the property
that $-\frac{i}{2\pi}\partial\bar{\partial}G(x,\cdot)=\delta_{x}-\omega_{0},$
where $\omega_{0}$ is the normalized curvature of the Fubini-Study
metric. Hence, the first variation of the functional $\mu\mapsto\int_{X}W\mu\otimes\mu$
on $\mathcal{P}(X)$ coincides with the first variation of $E(\mu)$
(formula \ref{eq:first var of E}), which proves formula \ref{eq:energy on P one is integral W}.

\subsection{\label{subsec:Conclusion-of-the}Conclusion of the proof of Theorem
\ref{thm:log Fano curve text} }

Set $p=-t$ and observe that 
\[
\int_{X}e^{-pW(x,y)}\mu_{0}(y)=\int_{X}e^{-\left(td_{L}\log\left\Vert x-y\right\Vert +\sum_{i}w_{i}\log\left\Vert x-p_{i}\right\Vert ^{2}\right)}dV_{X}
\]
For any given $y\in X$ the function $e^{-c\log\left\Vert x-y\right\Vert ^{2}}$
is locally integrable on $X$ iff $c<1.$ Hence, the right hand side
above is integrable iff for any fixed index $i$ 
\[
td_{L}/2+w_{i}<1,\,\,\,\forall i
\]
But this condition holds for some $t>1$ iff 
\[
d_{L}/2+w_{i}<1,\,\,\forall i
\]
 i.e. iff $1-\sum w_{j}/2+w_{i}<1$ for all $i,$ that is, $w_{i}<\sum_{j\neq i}w_{j},$
which is equivalent to the weight condition \ref{eq:weight condition}.
Hence, if the weight condition holds, then by Theorem \ref{thm:LDP two-point interaction},
the desired LDP follows.

Next, assume that the weight condition is violated. Without loss of
generality we may assume that it is violated for the index $i=1,$
which equivalently means that
\[
-d_{L}+2(1-w_{1})=0
\]
Set $B_{R}:=\{\left\Vert x-p_{1}\right\Vert \leq R\}.$ Since $e^{-\log\left\Vert x-y\right\Vert }\geq R^{-1}$
on $B_{R}$ we have 
\[
\int_{B_{R}^{N}}e^{W(x,y)}\mu_{0}(y)\geq(R^{-1})^{d_{L}N}\int_{B_{R}^{N}}\mu_{0}^{\otimes R}
\]
Using $\int_{|z|\leq R}e^{-w\log|z|^{2}}d(r^{2})\wedge d\theta=\frac{1}{1-w}(R^{2})^{1-w},$
we thus get
\[
\int_{B_{R}}\mu_{0}\geq\int e^{-\left(w_{1}\log\left\Vert x-p_{1}\right\Vert ^{2}\right)}dV_{X}\geq C(R^{2})^{(1-w_{1})}
\]
for some constant independent of $R.$ All in all, this means that
\[
\left(\int_{B_{R}^{N}}e^{W(x,y)}\mu_{0}(y)\right)^{1/N}\geq CR{}^{-d_{L}+2(1-w_{1})}\geq CR^{0}\geq C>0
\]
But the right hand side is independent of $R.$ Hence, letting $R\rightarrow0$
shows that the density $e^{W(x,y)}$ can not be in $L^{1}(X^{N}\mu_{0}^{\otimes N}),$
which means that $Z_{N,-1}=\infty,$ as desired. 

\subsection{The case of a general divisor $\Delta$}

Now consider the case of general coefficients $w_{i}\in]-\infty,1[.$
By the previous theorem $Z_{N,-1}$ diverges for large $N,$ unless
the weight condition \ref{eq:weight condition} holds. But fixing
any continuous metric $\left\Vert \cdot\right\Vert $ on $L$ we can
consider the corresponding probability measures $\mu_{\Delta,\beta}^{(N)},$
defined by formula \ref{eq:Gibbs meas}, which are well-defined when
$-\beta$ is sufficiently small. 
\begin{thm}
\label{thm:-P one with arb w}$\mathcal{Z}_{N}(\beta)<\infty$ iff
$\beta>-\gamma_{N}$ where 
\[
\gamma_{N}=\frac{N-1}{N}2\frac{1-\max_{i}w_{i}}{2-\sum_{i}w_{i}}
\]
Moreover, if $\mathcal{Z}_{N}(\beta)<\infty,$ then then the law of
the random variable $\delta_{N}$ on $(X^{N},\mu_{\Delta,\beta}^{(N)})$
satisfies a LDP with speed $N$ and rate functional $F_{\beta}-\inf_{\mathcal{P}(X)}F_{\beta}$
\end{thm}

\begin{proof}
First consider the case when $\left\Vert \cdot\right\Vert $ is the
metric $\left\Vert \cdot\right\Vert _{FS}$ induced from the Fubini-Study
metric on $\mathcal{O}(1).$ Then we get, as above, that $\mu_{\beta}^{(N)}=\mu^{(N)}[p]$
for $p=\beta\frac{N-1}{N}.$ Hence, by the argument in the beginning
of the previous section the integrability threshold is given by
\[
\gamma_{N}=\frac{N-1}{N}\gamma,\,\,\,\gamma=\sup\{t:\,td_{L}/2+w_{i}<1,\,\,\forall i\}=2\frac{1-\max_{i}w_{i}}{2-\sum_{i}w_{i}}.
\]
and the LDP follows from the general LDP in Theorem \ref{thm:LDP two-point interaction}.
Finally, writing a general continuous metric $\left\Vert \cdot\right\Vert $
as $e^{-u/2}\left\Vert \cdot\right\Vert _{FS}$ for a continuous function
$u$ on $X$ we can express $\mu_{\beta}^{(N)}=\mu^{(N)}[p]$ where
$\mu_{0}=e^{-(\beta+1)u}dV$ and again apply Theorem \ref{thm:LDP two-point interaction}. 
\end{proof}
As recalled in Section \ref{subsec:The-case beta neg} any minimizer
$\omega_{\beta}$ of $F_{\beta}$ satisfies the twisted Kähler-Einstein
equation \ref{eq:twisted KE with beta} with $\omega_{0}$ equal to
the normalized curvature form of the metric $\left\Vert \cdot\right\Vert $
on $L.$ 
\begin{rem}
In the case when $\Delta$ is trivial (i.e. $w_{i}=0)$ the formula
for $\gamma_{N}$ in the previous theorem was shown in \cite[Section 3]{fu},
using a different algebro-geometric argument. 
\end{rem}

\subsection{The zero-free hypothesis in the case of three points and the complex
Selberg integral}

We will next give an alternative proof of Theorem \ref{thm:log Fano curve text}
in the case when $m=3$ using the approach in Section \ref{subsec:Deforming-the-divisor}.
To simplify the notation we will drop the subscript $k$ in the notation
$N_{k}$ in formula \ref{eq:N k in terms of d L}. In other words,
as our data we take a divisor $\Delta$ on $\P_{\C}^{1}$ and an integer
$N$ which is strictly greater than one ($k$ can then be recovered
from formula \ref{eq:N k in terms of d L}). First recall that, by
Lemma \ref{lem:standard id}, the normalizing constant $\mathcal{Z}_{N}$
- that we will write as $\mathcal{Z}_{N}(\Delta)$ to indicate the
dependence on $\Delta$ - may be expressed 
\[
\mathcal{Z}_{N}(\Delta)=\int_{\C^{N}}\left(\prod_{i\neq j}\left|z_{i}-z_{j}\right|\right)^{-\frac{d_{L}}{N-1}}\prod_{i\leq N,j\leq m}\left|z_{i}-p_{j}\right|^{-2w_{i}}\prod_{i}\frac{i}{2}dz_{i}\wedge d\bar{z}_{i}.
\]
Now specialize to $m=3.$ Then we may, after perhaps applying an automorphism
of $\P_{\C}^{1},$ assume that the points $p_{1},p_{2}$ and $p_{3}$
are given by the points $0,$ 1 and $\infty.$ Hence, 
\[
\mathcal{Z}_{N}(\Delta)=\int_{\C^{N}}\left(\prod_{i\neq j}\left|z_{i}-z_{j}\right|\right)^{-\frac{d}{N-1}}\prod_{i}\left|z_{i}\right|^{-2w_{0}}\prod_{i}\left|z_{i}-1\right|^{-2w_{1}}\prod_{i}\frac{i}{2}dz_{i}\wedge d\bar{z}_{i},\,\,\,d=2-(w_{0}+w_{1}+w_{2}).
\]
 This integral is known as the \emph{complex Selberg integral} (when
expressed in terms of the parameters $w_{0},w_{1}$ and $d/(N-1)$).
The original Selberg integral is the integral obtained by replacing
$\C^{N}$ with $[0,1]^{N}$ and generalizes Euler's classical Beta-function
to $N>1$ (see the survey \cite{f-w}). Its complex version above
seems to first have appeared in Conformal Field Theory (CFT), in the
context of minimal CFTs, where it is known as one of the\emph{ Dotsenko-Fateev
integrals} \cite{d-f} (an equivalent formula was also established
in \cite{aa}, expressed in terms of the original Selberg integral).
By \cite[Formula B.9]{d-f} the integral $\mathcal{Z}_{N}(\Delta)$
is explicitly given by the following remarkable formula involving
the classical $\Gamma-$function 
\begin{equation}
\mathcal{Z}_{N}(\Delta)=N!\left(\frac{\pi}{l(-\frac{1}{2}\frac{d}{N-1})}\right)^{N}\prod_{j=1}^{N}\frac{l(-\frac{j}{2}\frac{d}{N-1})}{l(w_{1}+\frac{j}{2}\frac{d}{N-1})l(w_{2}+\frac{j}{2}\frac{d}{N-1})l(w_{3}+\frac{j}{2}\frac{d}{N-1})},\,\,\,\,l(x):=\frac{\Gamma(x)}{\Gamma(1-x)}.\label{eq:Z N as gamma prod}
\end{equation}

\begin{rem}
The integral $\mathcal{Z}_{N}(\Delta)$ also appears in connection
to the DOZZ-formula of Dorn-Otto and Zamolodchikov-Zamolodchikov for
the 3-point structure constants $C_{\gamma}(\alpha_{1},\alpha_{2},\alpha_{3})$
in Liouville CFT, which has recently been given a rigorous proof in
\cite{K-R-V} (see also the exposition in \cite[Section 2.3]{varg}).
A general formula for Selberg type integrals over a local field $F$
of characteristic zero was recently established in \cite{fu-zhu}
(specializing to Selberg's original integral when $F=\R_{>0}$ and
its complex generalization when $F=\C).$
\end{rem}

We next observe that for any given $\epsilon\in]0,1[$ $\mathcal{Z}_{N}(\Delta)$
is zero-free in the convex tube domain $\Omega$ in $\C^{3}$ defined
by 
\begin{equation}
\Omega=\{\boldsymbol{w}\in\C^{3}:\,\text{\ensuremath{\Re}}w_{i}<1,\,\,\text{\ensuremath{\Re}}w_{1}+\text{\ensuremath{\Re}}w_{2}+\text{\ensuremath{\Re}}w_{3}>0\}\label{eq:def of domain Omega}
\end{equation}

Indeed, by formula \ref{eq:Z N as gamma prod},
\[
\mathcal{Z}_{N}(\Delta)=N!\pi^{N}\left(\frac{\Gamma(1+\frac{1}{2}\frac{d}{N-1})}{\Gamma(-\frac{1}{2}\frac{d}{N-1})}\right)^{N}\prod_{j=1}^{N}\left(\frac{\Gamma(-\frac{j}{2}\frac{d}{N-1})}{\Gamma(1+\frac{j}{2}\frac{d}{N-1})}\frac{\Gamma(1-w_{1}-\frac{j}{2}\frac{d}{N-1})}{\Gamma(w_{1}+\frac{j}{2}\frac{d}{N-1})}\cdots\right),
\]
 where the dots indicator similar factors obtained by replacing $w_{1}$
with $w_{2}$ and $w_{3}$. It is a classical fact that $\Gamma(x)$
is a meromorphic zero-free function of $x\in\C$ with poles at $0,-1,-2,...$
Hence, the zeros of $\mathcal{Z}_{N}(\Delta)$ can only come from
the poles of the Gamma-factors appearing in the denominators above.
First consider the case when $d\neq0.$ Since $N\geq2$ and $2>\text{\ensuremath{\Re}}d$
the factor $\Gamma(-\frac{1}{2}\frac{d}{N-1})$ has no poles in $\Omega.$
Similarly, since $\text{\ensuremath{\Re}}d>-1$ the factor $\Gamma(1+\frac{j}{2}\frac{d}{N-1})$
has no poles and since $\text{\ensuremath{\Re}}w_{1}<1$ the factor
$\Gamma(w_{1}+\frac{j}{2}\frac{d}{N-1})$ has no poles in $\Omega$
(using that, for $\boldsymbol{w}\in\R^{3},$ when $d<0,$ $w_{1}+\frac{j}{2}\frac{d}{N-1}$
is minimal when $j=N$ and $N=2$ i.e. the minimum is $w_{1}+d=2-w_{1}-w_{2}>0)$
and likewise when $w_{1}$ is replaced by $w_{2}$ and $w_{3}.$ Finally,
when $d=0$ we get 
\[
\mathcal{Z}_{N}(\Delta)=N!\pi^{N}\left(\frac{\Gamma(1-w_{1})}{\Gamma(w_{1})}\cdots\right)^{N}
\]
which is non-zero, since $\Re w_{i}>0$ (and thus the denominator
above has no poles). 

This argument also reveals that the ``first'' negative poles of
$\mathcal{Z}_{N}(\Delta)$ appear when $1-x=0,$ for $x=w+td/2$ for
$w\in\{w_{0},w_{1},w_{2}\}$ and $t=i/(N-1)$ for $i=1,...,N,$ i.e.
when $w+td/2=1.$ In particular, if $w+td/2>1$ for the maximal value
of $t,$ i.e for $t=N/(N-1),$ then $\mathcal{Z}_{N}(\Delta)<\infty.$
This is precisely the condition for the finiteness of $\mathcal{Z}_{N}(\Delta)$
that came up in the beginning of Section \ref{subsec:Conclusion-of-the}
which is equivalent to the weight condition \ref{eq:weight condition}
for $w$ real. The explicit formula \ref{eq:Z N as gamma prod} for
$\mathcal{Z}_{N}(\Delta)$ then also gives 
\[
\mathcal{Z}_{N}(\Delta)\leq C^{N}.
\]

\subsubsection{Proving Theorem \ref{thm:log Fano curve text} by deforming $\Delta$
in the case when $m=3$}

We finally explain how to given an alternative proof of Theorem \ref{thm:log Fano curve text}
in the case $m=3$ using the zero-free property and the bound on $\mathcal{Z}_{N}(\Delta)$
established in the previous section, combined with the approach discussed
in Section \ref{subsec:Deforming-the-divisor}. In this case the affine
space $\mathcal{A}$ of all ``admissible'' $(s,\boldsymbol{w})$
is defined by the condition
\[
d_{L}^{-1}\left(2-(\sum_{i=1}^{m}w_{i})\right)=s,
\]
 where, as before, $d_{L}$ denotes the degree of the anti-canonical
line bundle of the given log Fano variety (whose weight vector is
denoted by $\boldsymbol{w}_{0}$ in Section \ref{subsec:Deforming-the-divisor}).
In particular, since we consider the case when $m=3$ we get $s<0$
by choosing a real weight vector $\boldsymbol{w}_{1}$ with components
sufficiently close to $1$ (which can be done as soon as $m>2)$ and,
in particular, $\boldsymbol{w}_{1}\in\Omega$ (where $\Omega$ is
the domain in formula \ref{eq:def of domain Omega}). Since the components
$p_{1},...,p_{m}$ of $\Delta$ are, trivially, non-singular and mutually
non-intersecting the implicit function theorem does apply. Hence,
so does the approach in Section \ref{subsec:Deforming-the-divisor}.

\section{\label{sec:Speculations-on-the}Speculations on the strong zero-free
hypothesis, L-functions and arithmetic geometry}

In this last section we discuss some intriguing relations between
the strong zero-free hypothesis for the partition functions $\mathcal{Z}_{N}(\beta)$
on Fano manifolds introduced in Section \ref{subsec:The-strong-zero-free}
and the zero-free property of the representation-theoretic (automorphic)
local zeta functions $L_{p}(s)$ appearing in the Langlands program
\cite{la1} . Conjecturally, the latter zeta functions are related
to arithmetic/motivic L$-$functions \cite{la2}. 

First recall that given a reductive group $G$ over a global field
$F$ together with automorphic representations $\pi$ and $\rho$
of $G$ and its Langlands dual, respectively, one attaches a local
L-function $L_{p}(s)$ to any place (prime) $p$ of $F.$ By definition,
the places $p$ of $F$ correspond to multiplicative (normalized)
absolute value $\left|\cdot\right|_{p}$ on $F.$ In the case when
$\left|\cdot\right|_{p}$ is non-Archimedean the local $L-$function
$L_{p}(s)$ is defined as the inverse of a characteristic polynomial
attached to the induced representation of $G_{p}$ and thus $L_{p}(s)$
is automatically zero-free. For Archimedean $\left|\cdot\right|_{p}$
the local L-function $L_{p}(s)$ may be defined as an appropriate
product of $\Gamma-$functions and is thus also zero-free; see \cite[Section 4]{kn}
for the case $G=GL(N,\C)$ and the relation to the local Langlands
correspondence. Conjecturally, any local automorphic L-function $L_{p}(s)$
is a product of the \emph{standard L-functions} corresponding to the
case when $G=GL(N,F_{p})$ and $\rho$ is the standard representation
of $GL(N,\C)$ \cite{la1} (generalizing the local versions of the
classical Hecke L-functions, e.g. the Riemann zeta function when $N=1$). 

\subsection{\label{subsec:Zeta-integral-expressions for the standard}The ``minimal''
partition function on $\P_{\C}^{n}$ as a standard local $L-$function }

In the standard case it was shown in \cite{g-j} (generalizing Tate's
thesis \cite{ta} to $N>1$) that $L_{p}(s)$ may - for any given
admissible irreducible representation $\pi$ - be realized as a ``zeta
integral'': 
\begin{equation}
L_{p}(s)=\int_{GL(N,F_{p})}|\det(g)|_{p}^{s}\mu_{p}(g)\label{eq:L as zeta integr}
\end{equation}
 for a distinguished measure $\mu_{p}$ on $GL(F_{p},N),$ depending
on $\pi,$ which is absolutely continuous with respect to Haar measure.
As a consequence, for such particular measures $\mu_{p}(g)$ the zeta
integral above is zero-free (since $L_{p}(s)$ is).

To see the relation to the partition functions $\mathcal{Z}_{N}(\beta)$
for Fano manifolds first note that we may, in the zeta integral above,
replace the group $GL(F_{p},N)$ with the algebra $\text{Mat}(F_{p},N)$
of $N\times N$ matrices $A$ with coefficients in $F_{p}$ (since
$\mu_{P}$ puts no mass on the complement of $GL(F_{p},N)$ in $M(F_{p},N)$).
Then, after a suitable shift, $s\rightarrow s+\lambda,$ the measure
$\mu_{p}$ is of the form 
\[
\mu_{p}=f_{\pi}\Phi dA,
\]
 where $dA$ is the additive Haar measure on $\text{Mat}(F_{p},N),$
the function $f_{\pi}$ is an appropriate matrix element of $\pi$
and $\Phi$ is a suitable Schwartz-Bruhat function on $\text{Mat}(F_{p},N)$.
In the ``unramified case'' $f_{\pi}$ is the spherical function
attached to $\pi$ and $\Phi$ it its own Fourier transform \cite[Prop. 6.12]{g-j}.
In case when $p$ is non-Archimedean this means that $\Phi$ is the
characteristic function of $M(O_{p},N),$ where $O_{p}$ denotes the
ring of integers of $F_{p},$ while in the Archimedean case $\Phi$
is the Gaussian (see \cite{ish} for the case $F_{p}=\C).$ Now, when
$p$ is taken to be the standard (squared) Archimedean absolute value
on $\C(=F_{p}),$ with $\pi$ the trivial representation, we get

\begin{equation}
\mathcal{Z}_{N}(\beta)=c_{n}\left(\Gamma\left(s+n+1\right)\right)^{-(n+1)}L_{p}(s),\,\,\,s=\beta(n+1)\label{eq:Z N as L factor}
\end{equation}
where $\mathcal{Z}_{N}(\beta)$ denotes the partition function for
the standard Kähler-Einstein metric on the Fano manifold $\P_{\C}^{n}$
with $N$ the minimal one (i.e. $N=n+1)$ considered in Example \ref{exa:P n}.
Indeed, this follows directly from combining formula \ref{eq:L as zeta integr}
(for $f_{\pi}=1)$ with formula \ref{eq:Z beta as matrix integral in appendix}
for $\mathcal{Z}_{N}(\beta)$ in the appendix. Note that the first
factor in the right hand side above is non-vanishing when $\Re\beta>-1$
and thus the zero-free property of $\mathcal{Z}_{N}(\beta)$ in the
strip $\Re\beta>-1$ can  be attributed to the zero-free property
of the corresponding local L-function $L_{p}(s).$

\subsection{Zeta integrals associated to Calabi-Yau subvarieties of $\text{Mat}(N_{k},\C)$}

It would be interesting to compute $\mathcal{Z}_{N_{k}}(\beta)$ in
more examples to check if it can be expressed as products (and quotients)
of Gamma-function and related to local Archimedean L-functions as
above. For example, if a reductive group $G$ acts holomorphically
on $X$ (e.g. if $X$ is a flag variety) one might be able to exploit
that the section $\det S^{(k)}$ over $X^{N_{k}}$ is invariant under
the diagonal action of $G$ on $X^{N_{k}},$ up to multiplication
by the determinant of the induced $G-$action on $H^{0}(X,-kK_{X}).$ 

For a general Fano manifold $X$ and $N_{k}$ it seems, however, unlikely
that $\mathcal{Z}_{N_{k}}(\beta)$ can be related to an automorphic
local $L-$function. Anyhow, as next explained the integral $\mathcal{Z}_{N_{k}}(\beta)$
can be expressed in terms of an integral over a Calabi-Yau subvariety
of $\text{Mat}(N_{k},\C),$ which has some intriguing structural similarities
with the zeta integral for the standard L-function $L_{p}(s)$ in
formula \ref{eq:L as zeta integr}. We start by lifting the integral
$\mathcal{Z}_{N_{k}}(\beta)$ to an integral where the projective
variety $X$ is replaced by the affine variety $Y_{k}$ of dimension
$n+1$ obtained by blowing down of the zero-section in the total space
of the line bundle $-kK_{X}\rightarrow X.$ To this end first note
that the standard $\C^{*}-$action on $-kK_{X}$ induces a $\C^{*}-$action
on the affine variety $Y_{k}$ with a unique fixed point $y_{0},$
i.e. $Y_{k}$ can be viewed as an affine cone over $X:$ 
\[
X\simeq\left(Y_{k}-\{y_{0}\}\right)/\C^{*}
\]
On the affine variety $Y_{k}$ there is a unique $\C^{*}-$equivariant
holomorphic top form $\Omega$ (modulo a multiplicative constant).
The Kähler-Einstein metric $\omega_{KE}$ on $X$ corresponds to a
conical Calabi-Yau metric $\omega_{CY}$ on $Y_{k},$ i.e. a Ricci-flat
Kähler metric with a conical singularity at $y_{0}$ \cite{g-m-s-y}.
Denote by $r$ the distance to the fixed point $y_{0}$ in $Y_{k}$
with respect to the Calabi-Yau metric $\omega_{CY}.$ We may then
express
\[
\mathcal{Z}_{N_{k}}(\beta)=c_{n}\left(\Gamma\left((n+1)\beta+n+1\right)\right)^{-N_{k}}\widetilde{\mathcal{Z}}_{N_{k}}(\beta),\,\,\,\,\,\,\widetilde{\mathcal{Z}}_{N_{k}}(\beta):=\int_{Y_{k}^{N_{k}}}\left|\det\Psi^{(k)}\right|^{2\beta/k}(e^{-r^{2}}\Omega\wedge\bar{\Omega})^{\otimes N_{k}},
\]
 where $\Psi^{(k)}$ is the holomorphic function on $Y_{k}^{N_{k}}$
corresponding to the section $\det S^{(k)}$ of $-kK_{X^{N_{k}}}$
and $c_{n}$ is a (computable) positive constant $c_{n}.$ This is
shown essentially as in the proof of Prop \ref{prop:formula in ex}
in the appendix. Next, assume that $k$ is sufficiently large to ensure
that $-kK_{X}$ is very ample. Then one obtains a holomorphic $(\C^{*})^{N_{k}}-$equivariant
embedding
\[
Y_{k}^{N_{k}}\rightarrow\text{Mat}(N_{k},\C),\,\,\,(y_{1},...,y_{N_{k}})\mapsto\left(\boldsymbol{\Psi}{}^{(k)}(y_{1}),....,\boldsymbol{\Psi}{}^{(k)}(y_{N_{k}})\right),
\]
where $\boldsymbol{\Psi}{}^{(k)}(y)$ denotes the $N_{k}-$tuple of
holomorphic functions $\psi_{1}^{(k)},...,\psi_{N_{k}}^{(k)}$ on
$Y_{k}$ corresponding to the fixed bases in $H^{0}(X,-kK_{X}).$
In geometric terms the embedding above is just the embedding induced
from the Kodaira embedding of $X$ in the projectivization of $H^{0}(X,-kK_{X})^{*}.$
Denoting by $\mathcal{Y}_{k}$ the image of $Y_{k}^{N_{k}}$ in $\text{Mat}(N_{k},\C)$
we can thus express $\widetilde{\mathcal{Z}}_{N_{k}}(\beta)$ as a
matrix integral:
\[
\widetilde{\mathcal{Z}}_{N_{k}}(\beta):=\int_{\mathcal{Y}_{k}\Subset\text{Mat}(N_{k},\C)}\left|\det A\right|^{2\beta/k}e^{-r^{2}}\Omega\wedge\bar{\Omega},
\]
where now $r$ denotes the distance to the origin in $\text{Mat}(N_{k},\C)$
with respect to the Calabi-Yau metric on the subvariety $\mathcal{Y}_{k}$
and $\Omega$ denotes the equivariant holomorphic top form on $\mathcal{Y}_{k}$
(which can be viewed as a Poincaré type residue of the standard holomorphic
top form on $\text{Mat}(N_{k},\C)$ along $\mathcal{Y}_{k}).$ This
matrix integral is reminiscent of the integral expression \ref{eq:L as zeta integr}
for the local L-functions $L_{p}(s),$ if $\mu_{p}$ is taken to be
the measure on $\text{Mat}(N_{k},\C)$ induced by pairing of $\Omega\wedge\bar{\Omega}$
with the subvariety $\mathcal{Y}_{k},$ weighted by the Gaussian type
factor $e^{-r^{2}}$ (and $s:=\beta/k$). In view of this structural
similarity it is tempting to speculate on a very strong zero-free
hypothesis, saying that, in general, the ``lifted'' partition function
$\widetilde{\mathcal{Z}}_{N_{k}}(\beta)$ is zero-free on all of $\C,$
when viewed as a meromorphic function.
\begin{rem}
\label{rem:orbi cy}The same considerations apply when $X$ is a Fano
orbifold if $K_{X}$ is replaced by the orbifold canonical line bundle
(coinciding with $-K_{X}+\Delta$ as $\Q-$line bundle). Then the
natural projection from $Y_{k}-\{y_{0}\}$ to $X$ is a submersion
over the complement of the branching divisor $\Delta$ and the orbifold
Kähler-Einstein metric on $X$ corresponds to a bona fide Calabi-Yau
metric on $Y_{k}-\{y_{0}\}$ \cite{g-m-s-y}. 
\end{rem}

One further piece of evidence for the very strong form of the zero-free
hypothesis (complementing the ``minimal'' case on $\P^{n}$ appearing
in Prop \ref{prop:formula in ex}) is provided by the case when $X=\P^{1}$
and $k=1,$ i.e. $N_{k}=3$ (which is the case next to minimal dimension,
$N_{k}=n+1$). Then, identifying $-K_{X}$ with $2\mathcal{O}(1)$
and $\det S^{(1)}$ with the Vandermonde determinant $D^{(3)}$ on
$\C^{3}$ (as in Lemma \ref{lem:standard id}) and using that the
Kähler-Einstein metric is explicitly given by the Fubini-Study metric
(formula \ref{eq:FS}), $\mathcal{Z}_{N_{k}}(\beta)$ may be expressed
as
\[
\mathcal{Z}_{N_{k}}(\beta)=\int_{\C^{3}}\prod_{i<j\leq3}\left|z_{i}-z_{j}\right|^{2\beta}\prod_{i<j\leq3}\left(1+\left|z_{i}\right|^{2}\right)^{-(2\beta+2)},
\]
 integrating with respect to Lebesgue measure. Applying formula in
\cite[Thm 1]{v-s} (to $\sigma_{i}=\nu_{i}=\beta+1),$ which originally
appeared in Conformal Field Theory, thus yields
\begin{equation}
\mathcal{Z}_{N_{k}}(\beta)=\pi^{3}\left(\Gamma\left(2\beta+2\right)\right)^{-3}\Gamma(3\beta+2)\Gamma(\beta+1)^{3}.\label{eq:Z N when N is three}
\end{equation}
This means that the meromorphic function $\widetilde{\mathcal{Z}}_{N_{k}}(\beta)$
is a product of four Gamma functions and thus zero-free on all of
$\C.$ The elegant proof in \cite{v-s} leverages the diagonal action
of $GL(N_{k},\C)$ on $X^{N_{k}}$ alluded to above (following the
corresponding real case considered in \cite{b-r} in the context of
automorphic triple products). 

The general case on $X=\P^{1},$ when $N_{k}>3,$ appears to be open.
However, a similar formula does hold for any $N_{k}$ when $X$ is
replaced by its\emph{ real }points, i.e. when $\P_{\C}^{1}$ is replaced
by $\P_{\R}^{1}.$ Then the role of $\mathcal{Z}_{N_{k}}(\beta)$
is played by

\[
\mathcal{Z}_{N_{k}}(\beta)_{\R}:=\int_{(\P_{\R}^{1})^{N_{k}}}\left\Vert \det S^{(k)}\right\Vert ^{\beta/k}dV^{\otimes N_{k}}=\int_{(S^{1})^{N_{k}}}\prod_{i<j\leq N_{k}}\left|z_{i}-z_{j}\right|^{\frac{2\beta}{N_{k}-1}}d\theta^{\otimes N_{k}},\,\,\,\,N_{k}=2k+1
\]
where $\left\Vert \cdot\right\Vert $ denotes the Fubini-Study metric
and $dV$ denotes the corresponding volume form on $(\P_{\R}^{1}).$
In the second equality above we have exploited that the integrand
is invariant under the diagonal action of $SU(2)$ to replace the
real points $\P_{\R}^{1}$ of $\P_{\C}^{1}$ with the unit-circle
$S^{1}$ in $\C\subset\P_{\C}^{1}.$ The latter integral over $(S^{1})^{N_{k}}$
coincides with the partition function for the 2D Coulomb gas confined
to $S^{1}\subset\C$ at inverse temperature $2\beta/(N_{k}-1)$ (known
as the circular ensemble). Applying \cite[formula 1.12]{f-w} (originally
conjectured by Dyson and established by Gunson and Wilson) thus yields
\[
\mathcal{Z}_{N_{k}}(\beta)_{\R}=(2\pi)^{N_{k}}\Gamma(1+\beta\frac{1}{N_{k}-1})^{-N_{k}}\Gamma(1+\beta\frac{N_{k}}{N_{k}-1}),\,\,\,\,N_{k}=2k+1
\]
 This formula reveals that the real analog $\mathcal{Z}_{N_{k}}(\beta)_{\R}$
of the partition function on $\P_{\C}^{1}$ does satisfy the strong
zero-free hypothesis. This real analog may - from the point of view
of localization - be obtained by replacing the squared absolute value
$\left|\cdot\right|_{\C}^{2}$ corresponding to the complex Archimedean
place of the global field$\Q$ with the  absolute value $\left|\cdot\right|_{\R}$
corresponding to the real Archimedean place of $\Q.$ The extension
to non-Archimedean places is discussed in Section \ref{subsec:Extension-to-non-Archimedean}.
But first a brief detour on arithmetic aspects of the partition function.

\subsection{\label{subsec:Invariants-of-arithmetical}Invariants of arithmetical
Fano varieties}

Let $\mathcal{X}$ be an arithmetic variety of dimension $n+1$ (i.e.
a projective scheme flat over $\Z,$ $\mathcal{X}\rightarrow\text{Spec \ensuremath{\Z})}$
such that the corresponding $n-$dimensional complex variety $X$
(i.e. the complexification of the generic fiber $X_{\Q}$ of $\mathcal{X})$
is Fano. Assume that $\mathcal{X}$ is endowed with a relatively nef
line bundle $\mathcal{L}$ such that the induced line bundle on $X$
equals $-K_{X}.$ Then $(\mathcal{X},\mathcal{L})$ induces a section
$\det S^{(k)}$ of $-kK_{X^{N_{k}}}\rightarrow X^{N_{k}}$ which is
uniquely determined up to multiplication by $\pm1$. Indeed, $(\mathcal{X},\mathcal{L})$
induces a lattice $H^{0}(\mathscr{X},k\mathscr{L})$ of integral sections
in $H^{0}(X,-kK_{X})$ and $\det S^{(k)}$ may be defined as in in
formula \ref{eq:slater determinant} with respect to any basis in
$H^{0}(\mathscr{X},k\mathscr{L})$ (any two such bases are related
by a matrix with integral coefficients, which thus has determinant
equal to $\pm1$). As a consequence, the corresponding partition function
$\mathcal{Z}_{N_{k}}(\beta)$ only depends on $(\mathcal{X},\mathcal{L})$
and the choice of a metric $\left\Vert \cdot\right\Vert $ on $-K_{X}$
(and is independent of the metric at $\beta=-1).$ In fact, the explicit
expression for $\mathcal{Z}_{N_{k}}(\beta)$ appearing Prop \ref{prop:formula in ex}
- related to a local L-function in formula \ref{eq:Z N as L factor}
- was computed with respect to the standard integral model $(\mathcal{X},\mathcal{L})$
for $(\P^{n},\mathcal{O}(1))$ (where $H^{0}(\mathscr{X},k\mathscr{L})$
is the lattice spanned by the sections defined by multinomials). In
the light of the speculations in the previous section this appears
to fit well with the arithmetical side of the Langlands program. 

In particular, taking $\beta=-1$ yields an invariant $\mathcal{Z}_{N_{k}}$
of $(\mathcal{X},\mathcal{L})$ (which is finite iff $X$ is Gibbs
stable at level $k$). The following conjecture relates the arithmetic
invariants $\mathcal{Z}_{N_{k}}$ to the arithmetic intersection numbers
introduced by Gillet-Soulé in the context of Arakelov geometry (see
the book \cite{so}).
\begin{conjecture}
\label{conj:arithm}Let $(\mathcal{X},\mathcal{L})$ be an arithmetic
variety as above and assume that the corresponding Fano manifold $X$
admits a unique Kähler-Einstein metric, whose volume form is denoted
by $dV_{KE},$ normalized to have unit total volume. Then, as $k\rightarrow\infty,$
$\frac{(n+1)!}{k^{n}}\log\mathcal{Z}_{N_{k}}$ converges towards the
$(n+1)-$fold arithmetic self-intersection number of the line bundle
$\mathcal{L},$ metrized by $dV_{KE}.$
\end{conjecture}

In fact, using the arithmetic Hilbert-Samuel theorem in \cite[Thm 1.4]{Zh0}
(generalizing the relative ample case in \cite{g-s}) this conjecture
is equivalent to the convergence of the partition function appearing
in Theorem \ref{thm:equiv cond for conv}, defined with respect to
any basis of $H^{0}(X,kK_{X})$ which is orthonormal with respect
to the Hermitian product induced by a Kähler metric on $X.$ Thus,
by Theorem \ref{thm:zero-free}, in order to establish the conjecture
it would, for example, be enough to show that the lifted partition
function $\widetilde{\mathcal{Z}}_{N_{k}}(\beta)$ may be expressed
as a product of $O(N_{k})$ shifted Gamma-functions all of whose poles
are located in the region where $\Re\beta<-1-\epsilon$ for some $\epsilon>0.$ 
\begin{rem}
Other \label{subsec:Extension-to-non-Archimedean} of (polarized)
arithmetic varieties on arithmetic varieties $\mathcal{X},$ endowed
with a relatively ample line bundle $\mathcal{L},$ are introduced
in \cite{bo2,zh} (which are finite precisely when $(X,k\mathcal{L})$
is Chow stable) and related to constant scalar curvature metrics in
\cite{o}.
\end{rem}

The analog of Conjecture \ref{conj:arithm} does hold when $-K_{X}$
is replaced by $K_{X}$ (assumed ample) and $\log\mathcal{Z}_{N_{k}}$
is replaced by the arithmetic invariant $-\log\mathcal{Z}_{N_{k}}$
(as follows from combining the convergence of $\mathcal{Z}_{N_{k}}(1)$
in Theorem \ref{thm:beta pos} with the arithmetic Hilbert-Samuel
theorem).

\subsection{Extension to non-Archimedean places }

In view of the connections to local L-functions $L_{p}$ at the (complex)
Archimedean place $p,$ exhibited in Section \ref{subsec:Zeta-integral-expressions for the standard},
one may wonder if the probabilistic setup can be extended to non-Archimedean
places $p?$ The case of the trivial place is discussed in \ref{subsec:Zeta-integral-expressions for the standard},
in connection to Gibbs stability. What follows are some speculations
on the case of non-trivial non-Archimedean places $p,$ inspired by
the adelic geometric setup in \cite{c-t} where geometric Igusa local
zeta functions are studied (see Section \ref{subsec:Archimedean-zeta-functions}).

Let $X$ be a non-singular variety defined over $\Q$ and first consider
the case when $K_{X(\Q)}$ is ample. Given a non-trivial non-archimedean
place $p$ (i.e a prime number) denote by $X(\Q_{p})$ the projective
variety over the corresponding $p-$adic local field $\Q_{p}$ (the
completion of $\Q$ with respect to $|\cdot|_{p}$), which comes with
the structure of a $\Q_{p}-$analytic manifold. By general principles,
any continuous metric on $K_{X(\Q_{p})}$ induces a measure on $X(\Q_{p})$,
which is absolutely continuous wrt the local Haar measures \cite[Section 2.1]{c-t}.
In particular, a section $s_{k}$ of $kK_{X(\Q_{p})}$ induces a measure
on $X(\Q_{p}),$ whose local density may be symbolically expressed
as $|s_{k}|_{p}^{1/k}$. \footnote{One can also consider a field extension $F_{p}$ of $\Q_{p}$ and
get a measure on the corresponding analytic manifolds $X(F_{p}),$
as in \cite{j-n}, but here $F_{p}=\Q_{p},$ for simplicity.} Hence, replacing the squared Archimedean absolute value appearing
in formula \ref{eq:canon prob measure intro} with $|\cdot|_{p}$
one arrives at a symmetric probability measure $\mu_{p}^{(N_{k})}$
on $X(\Q_{p})^{N_{k}}.$ This construction thus yields a canonical
random point process on $X(\Q_{p}).$ Accordingly, it seems natural
to ask if the convergence in Theorem \ref{thm:ke intro} can be extended
to this non-archimedean setup, if $dV_{KE}$ is replaced by an appropriate
measure $dV_{KE,p}$ on $X(\Q_{p})?$ In analogy with the archimedean
setup the measure $dV_{KE,p}$ should be characterized as the unique
minimizer of a free energy type functional $F_{1}$ on the space of
probability measure $\mu$ on $X(\Q_{p})$ of the form: 
\begin{equation}
F_{1}(\mu)=E(\mu)+\text{Ent}(\mu),\label{eq:free energy in p-adic}
\end{equation}
 where $\text{Ent}(\mu)$ denotes the entropy of the measure $\mu$
relative to a fixed measure on $X(\Q_{p}),$ absolutely continuous
wrt the local Haar measure and $E(\mu)$ is a non-archimedan analog
of the energy discussed in Section \ref{subsec:The-case beta pos}.
In particular, $dV_{KE,p}$ is then absolutely continuous wrt the
local Haar measure. 

Ideally, one might hope that the collection of metrics on $-K_{X(\Q_{p})}$
defined by $dV_{KE,p}$ - as $p$ ranges over all primes $p$ - is
induced by some model $(\mathscr{X},\mathscr{L})$ for $(X,K_{X(\Q)})$
over $\ensuremath{\Z},$ away from primes $p$ with bad reduction
(cf. \cite[Section 2.2.3]{c-t}). This would - loosely speaking -
yield a probabilistic construction of a ``canonical'' integral model
attached to $X(\Q).$ This is in line with the analogy between the
Kähler-Einstein condition of a metric on $X(\C)$ (i.e. at $p=\infty$)
and the minimality condition of an integral model for $X(\Q)$ put
forth in \cite{m} and further studied in \cite{o}. 
\begin{rem}
Embedding, $X(\Q_{p})$ in its Berkovich analytification $X_{p}^{an}$
and pushing forward a measure $\mu$ on $X(\Q_{p})$ to $X_{p}^{an}$
the functional on $C^{0}($$X_{p}^{an}$) defined as the Legendre-Fenchel
transform of the functional $E(\mu)$ in formula \ref{eq:free energy in p-adic}
should - in analogy to the archimedean setup \cite{bbgz,berm6} -
be given by the primitive of the non-Archimedean Monge-Ampère operator
introduced in \cite{k-t,ch}. The primitive in question is called
the ``energy functional'' in \cite{b-f-j}. In the case of a trivial
non-Archimedean absolute value such an energy $E(\mu)$ appears in
\cite[formula 6.1]{b-j0} and plays an important role in the non-Archimedean
approach to K-stability.
\end{rem}

Similar considerations apply in the Fano case. In particular, to a
given metric on $-K_{X(\Q_{p})}$ one can associate a lifted partition
function $\widetilde{\mathcal{Z}}_{N_{k},p}(\beta).$ By general principles
\cite[Section 4.1]{c-t}, this defines a meromorphic function on $\C$
which in the light of Section \ref{subsec:Zeta-integral-expressions for the standard}
plays the role of the local L-functions $L_{p}$ in the Langlands
program. More precisely, in order to render $\widetilde{\mathcal{Z}}_{N_{k},p}(\beta)$
as canonical as possible the metric on $-K_{X(\Q_{p})}$ should be
taken to be defined by a ``canonical'' integral model $(\mathscr{X},\mathscr{L})$
for $(X(\Q),-K_{(\Q)})$ and $\det S^{(k)}$ should be defined with
respect to any basis in $H^{0}(\mathscr{X},\mathscr{L})$ (as in Section
\ref{subsec:Invariants-of-arithmetical}). Finally, one could then
attempt to define a global L-type function as an Euler product of
$\widetilde{\mathcal{Z}}_{N_{k},p}(\beta)$ over all $p,$ generalizing
the Riemann zeta function. 

\section{\label{sec:Appendix:-log-canonical}Appendix: log canonical thresholds
and Archimedean zeta functions}

In this appendix we recall the basic notions of log canonical thresholds,
$\alpha-$invariants and their connections to Archimedean zeta functions,
which are as essentially well-known. We conclude with a proof of the
formula appearing in Example \ref{exa:P n}.

\subsection{Log canonical thresholds (lct)}

Let $X$ be a compact complex manifold.

\subsubsection{The lct of a divisor on $X$}

By definition an $\R-$divisor $D$ is a finite formal sum of irreducible
analytic subvarieties $D_{i}\subset X$ of complex codimension one:
\[
D=\sum_{i=1}^{m}c_{i}D_{i},\,\,\,c_{i}\in\R.
\]
 The\emph{ log canonical threshold} $\text{lct}_{X}(D)$ of an $\R-$divisor
$D$ has various algebro-geometric formulations (using discrepancies,
valuations, multiplier ideal sheaves,...) \cite{ko0}, but for the
purposes of the present paper it will be enough to recall its analytic
definition as an integrability threshold. First consider the case
when the coefficients $D$ are in $\Z_{+}.$ This equivalently means
that there exists a holomorphic line bundle $L_{D}\rightarrow X$
and a holomorphic section $s_{D}$ such that $D$ is cut-out by $s_{D},$
including multiplicities, i.e. $s_{D}$ vanishes to order $c_{i}$
along the irreducible varieties $D_{i}.$ The lct may then be defined
as the following integrability index:
\begin{equation}
\text{lct}_{X}(D):=\sup_{\gamma>0}\left\{ \gamma:\int_{X}\left\Vert s_{D}\right\Vert ^{-2\gamma}dV<\infty\right\} ,\label{eq:def of lct}
\end{equation}
 in terms of any Hermitian metric $\left\Vert \cdot\right\Vert $
on $L$ and volume form $dV$ on $X.$ This definition  first extends
to the case when $c_{i}\in\Z,$ if $s_{D}$ is viewed as a meromorphic
section, so that the negative coefficients correspond to the poles
of $s_{D},$ and then to $c_{i}\in\Q$ by viewing $s_{D}$ as a multi-valued
holomorphic section and noting that $\left\Vert s\right\Vert $ is
still a well-defined function on $X$ (taking values in $[0,\infty]$).
Finally, the definition extends, by continuity, to any $\R-$divisor
$D$ or, alternatively, by noting that the function $\left\Vert s_{D}\right\Vert $
is still well-defined (and can be viewed as the norm on an $\R-$line
bundle, i.e. a formal sum of the line bundles $L_{D_{i}}$).

\subsubsection{The lct of a divisor on $(X,\Delta)$}

More generally, if $\Delta$ is a given $\Q-$divisor of $X$ then
the log canonical threshold of $D$ relative to the\emph{ log pair
}$(X,\Delta)$ \cite{c-p-s} may be analytically defined as 
\[
\text{lct}_{(X,\Delta)}(D):=\sup_{\gamma>0}\left\{ \gamma:\,\int_{X}\left\Vert s\right\Vert ^{-2\gamma}dV_{\Delta}<\infty\right\} ,
\]
where $dV_{\Delta}$ is a measure on $X$ with singularities encoded
by $\Delta,$ i.e. locally $dV_{\Delta}$ may be expressed as 
\[
dV_{\Delta}=\left\Vert s_{\Delta}\right\Vert ^{-2}dV_{X}
\]
 for some bona fide volume form $dV_{X}$ on $X$ and metric $\left\Vert \cdot\right\Vert $
on the $\Q-$line bundle with multivalued holomorphic section $s_{\Delta}$
corresponding to $\Delta.$ More generally, as in the previous section
$\Delta$ may be taken to be an $\R-$divisor on $X.$

\subsubsection{The lct of a line bundle $L$ and the $\alpha-$invariant.}

The log canonical threshold $\text{lct}_{X}(L)$ of a line bundle
$L\rightarrow X$ is now defined by
\[
\text{lct}_{X}(L):=\inf_{D\sim L}\text{lct}_{X}(D),
\]
 where $D$ ranges over the divisors attached to all the many-valued
holomorphic section $s$ of $L.$ By \cite{dem} this coincides with
Tian's $\alpha-$invariant of $L:$ 
\begin{equation}
\alpha(L):=\sup_{\gamma>0}\left\{ \gamma:\,\exists C\,\int_{X}e^{-\gamma(\phi-\phi_{0})}dV\leq C\forall\phi\in\mathcal{H}(L)\right\} ,\label{eq:def of alpha of L}
\end{equation}
 where $\mathcal{H}(L)$ denotes the space of all metrics on $L$
with positive curvature and $\phi_{0}$ denotes a fixed smooth reference
metric on $L$ (using additive notation for metrics so that $\phi-\phi_{0}$
defines a function on $X.$ More generally, the log canonical threshold
$\text{lct}_{(X,\Delta)}(L)$ of a line bundle $L\rightarrow X$ wrt
a log pair $(X,\Delta)$ \cite{c-p-s} is defined by

\[
\text{lct}_{(X,\Delta)}(L):=\inf_{D\sim L}\text{lct}_{(X,\Delta)}(D).
\]
 This coincides with the $\alpha-$invariant defined wrt the log pair
$(X,\Delta)$ obtained by replacing $dV$ in formula \ref{eq:def of alpha of L}
with $dV_{(X,\Delta)},$ as shown the appendix of \cite{berm6}.

\subsection{\label{subsec:Archimedean-zeta-functions}Archimedean zeta functions}

Let $\mu_{0}$ be a measure on $\C^{n}$ with compact support and
$\psi\in L^{1}(\mu_{0}).$ Then we may define the integrability threshold
$\text{lct}_{\mu_{0}}(\psi)$ as in formula \ref{eq:def of lct},
by replacing $\log\left\Vert s\right\Vert ^{2}$ with $\psi$ and
$dV$ by $\mu_{0}.$ The integral
\[
Z(\beta)=\int_{\C^{n}}e^{2\beta\psi}\mu_{0},
\]
 defines a holomorphic function on the strip $\{\Re\beta>-\text{lct}_{\mu_{0}}(\psi)\}$
in $\C$ (using that, in this strip, $e^{\beta\psi}\in L^{1}(\mu_{0})$
and that the integrand is holomorphic in $\beta).$ In the case when
$\psi=\log|f|^{2}$ for $f$ holomorphic, or more precisely,
\begin{equation}
Z(\beta)=\int_{\C^{n}}|f|^{2\beta}\Phi dx,\label{eq:Z beta for f}
\end{equation}
for a Schwartz function $\Phi,$ the holomorphic function $Z(\beta)$
on the strip $\{\Re\beta>-\text{lct}_{\mu_{0}}(\psi)\}$ extends to
a meromorphic function in $\C,$ whose poles are located at the negative
real axes. 
\begin{rem}
\label{rem:igusa}This follows from classical results of Atiyah and
Bernstein, extended by Igusa to a more general setting of zeta function
attached to polynomials defined over local fields \cite{ig}. Briefly,
meromorphic functions $Z(\beta)$ of the form \ref{eq:Z beta for f}
can be defined more generally by replacing $\C$ and its standard
Archimedean absolute value $\left|\cdot\right|$ with any local field
$F,$ endowed with an absolute value $\left|\cdot\right|_{F}.$ Such
functions $Z(\beta)$ are usually called\emph{ Igusa local zeta function}
\cite{ig} and thus $Z(\beta)$ in formula \ref{eq:Z beta for f}
is called an Igusa Archimedean zeta function or simply an \emph{Archimedean
zeta function} in the literature on algebraic and arithmetic geometry.
In the case when $f$ is a polynomial with integer coefficents and
$F$ is the $p-$adic field, $F=\Q_{p},$ the meromorphic function
$Z(\beta)$ encodes the number of solutions of the equation $f(x_{1},...,x_{n})=0,$
modulo powers of $p,$ when $\Phi$ is taken as the characteristic
function of the $n-$fold product of the ring $\Z_{p}$ of integers
of $\Q_{p},$
\end{rem}

Similarly, given a holomorphic section $s$ of a line bundle $L\rightarrow X$
over a compact complex manifold, a metric $\left\Vert \cdot\right\Vert $
on $L$ and a singular volume form $dV_{\Delta}$ associated to a
log pair $(X,\Delta)$ 
\begin{equation}
\mathcal{Z}(\beta):=\int_{X}\left\Vert s\right\Vert ^{2\beta}dV_{(X,\Delta)}\label{eq:Z of beta as integral over s norm}
\end{equation}
 defines a holomorphic function in the strip $\{\Re\beta>-\text{lct}_{(X,\Delta)}(D)\}$
in $\C,$ where $D$ denotes the divisor cut out by the section $s.$
More precisely the function $\mathcal{Z}(\beta)$ extends to a meromorphic
function on $\C,$ whose poles are located on the negative real axes
(using a partition of unity to reduce to the case of $X=\C^{n}).$
The first negative pole is precisely $-\text{lct}_{(X,\Delta)}(D).$ 
\begin{rem}
Functions of the form \ref{eq:Z of beta as integral over s norm}
have previously appeared in a general adelic setup \cite{c-t} (containing
both the Archimedean and the $p$-adic setup), motivated by number
theory and arithmetic geometry on log Fano varieties.
\end{rem}

In the present probabilistic setup on Fano manifolds, discussed in
Section \ref{subsec:The-strong-zero-free}, the manifold is of the
form $X^{N_{k}}$$,$ the section is the many-valued holomorphic section
$(\det S^{(k)})^{1/k}$ of $-K_{X^{N_{k}}}$ and the measure is of
the form $dV_{X}^{\otimes N_{k}}$ (and similarly in the case of log
Fano pairs). We conclude by proving the explicit formula for $\mathcal{Z}(\beta)$
stated in Example \ref{exa:P n}.
\begin{prop}
\label{prop:formula in ex}In the setup of Example \ref{exa:P n}
the following formula holds 
\[
\mathcal{Z}(\beta)=c_{n}\frac{\prod_{j=1}^{n}\Gamma\left(\beta(n+1)+j\right)}{\left(\Gamma\left(\beta(n+1)+n+1\right)\right)^{n}}\cdot
\]
In particular, the maximal holomorphicity strip of $\mathcal{Z}(\beta)$
is given by $\Omega=\{\Re(\beta)>-1/(n+1)\}\Subset\C$ and $\mathcal{Z}(\beta)$
is zero-free in $\Omega.$ More precisely, the zeros of $\mathcal{Z}(\beta)$
are located at $\beta=-1+j/(n+1)$ where $j=0,1,2,....$
\end{prop}

\begin{proof}
In this ``minimal'' case a basis $s_{1},...,s_{N_{k}}$ in the complex
vector space $H^{0}(X,-kK_{X})=H^{0}(\P^{n},\mathcal{O}(1))$ is obtained
from the homogeneous coordinates $Z_{0},...,Z_{n}$ on $\P^{n}.$
Denote by $\boldsymbol{Z}:=(Z_{0},...,Z_{n})$ the corresponding vector
in $\C^{n+1}.$ We will represent an element in $(\boldsymbol{Z}_{1},....,\boldsymbol{Z}_{N})\in(\C^{n+1})^{N}$
with an $(n+1)\times N-$matrix, denoted by $[\boldsymbol{Z}].$ Then
the corresponding Slater determinant $\det S^{(k)}$ may be identified
with the homogeneous polynomial $\det[\boldsymbol{Z}]$ on $\C^{(n+1)^{2}},$
defined by the determinant of the matrix $[\boldsymbol{Z}].$ Using
the $SU(n+1)-$symmetry of the Fubini-Study metric on $\mathcal{O}(1)\rightarrow\P^{n}$
we may then first lift the integral $Z(\beta)$ on $(\P^{n})^{n+1}$
to the product of unit-spheres $S$ in $\C^{n+1}:$
\[
\mathcal{Z}(\beta)=c_{n}\int_{S^{(n+1)}}\left|\det[\boldsymbol{Z}]\right|^{2s}d\sigma^{\otimes N},\,\,\,s:=\beta/k
\]
where $d\sigma$ denotes the standard $SU(n+1)-$invariant measure
on $S.$ Next, exploiting that $\det[\boldsymbol{Z}]$ is homogeneous
of degree $1$ in each column, gives 
\[
\int_{S^{(n+1)}}\left|\det[\boldsymbol{Z}]\right|^{2s}d\sigma^{\otimes N}=c_{n}\frac{\int_{\C^{(n+1)^{2}}}\left|\det[\boldsymbol{Z}]\right|^{2s}e^{-|\boldsymbol{Z}|^{2}}d\lambda}{\left(\int_{0}^{\infty}(r^{2})^{s}e^{-r^{2}}r^{2(n+1)-1}dr\right)^{n+1}}.
\]
 Hence, making the change of variables $t=r^{2}$ in the denominator
(and rewriting $r^{2(n+1)-1}dr=r^{2(n+1)}r^{-2}d(r^{2})/2)$ reveals
that
\begin{equation}
\mathcal{Z}(\beta)=c_{n}\frac{\int_{\C^{(n+1)^{2}}}\left|\det[\boldsymbol{Z}]\right|^{2s}e^{-|\boldsymbol{Z}|^{2}}d\lambda}{\left(\Gamma\left(s+n+1\right)\right)^{(n+1)}},\,\,\,\Gamma(a):=\int_{0}^{\infty}t^{a}e^{-t}\frac{dt}{t}.\label{eq:Z beta as matrix integral in appendix}
\end{equation}
 Finally, the proof is concluded by invoking the following formula
in \cite[Thm 6.3.1]{ig}:
\begin{equation}
Z(s):=\int_{\C^{(n+1)^{2}}}\left|\det[\boldsymbol{Z}]\right|^{2s}e^{-|\boldsymbol{Z}|^{2}}d\lambda=c_{n}\prod_{j=1}^{n+1}\Gamma\left(s+j\right).\label{eq:form Z of s}
\end{equation}
\end{proof}
\begin{rem}
\label{rem:bernstein}The proof of formula \ref{eq:form Z of s} in
\cite{ig} exploits that the polynomial $f:=\det[\boldsymbol{Z}]$
on $\C^{(n+1)^{2}}$ has the property that 
\begin{equation}
P(\partial)f^{s+1}=b(s)f^{s}\label{eq:bernstein}
\end{equation}
with $b(s)=\prod_{j=1}^{n+1}(s+j),$ when $P(z)=f(z).$ This leads
to the functional relation $b(s)Z(s)=Z(s+1),$ that can then be compared
with the classical functional relation for $\Gamma(s)$ to deduce
formula \ref{eq:form Z of s}. Recall that in general, given a polynomial
$f(z)$ on $\C^{m},$ the monic polynomial $b(s)$ on $\C$ with minimal
degree for which there exists a polynomial $P(z)$ satisfying formula
\ref{eq:bernstein} is called the \emph{Bernstein-Sato polynomial
}attached to $f$ \cite{ig}. In general, it is very hard to compute
$b(s)$ explicitly (and thus to also to find $P(z))$ but the present
case, $f(z)=\det[\boldsymbol{Z}],$ fits into Sato's theory of prehomogenuous
vector spaces. This is explained in \cite{ig}. Alternatively, formula
\ref{eq:form Z of s} follows from the Iwasawa decomposition of $GL(N,\C)$
(as in \cite[Section 2]{ish}). It would be interesting to see if
similar considerations could be applied to $X=\P^{n}$ when $N_{k}$
is not assumed to be minimal, i.e. when $N_{k}>n+1.$ However, even
the case when $n=1$ appears to be open (apart from the case when
$N_{k}=3$ appearing in formula \ref{eq:Z N when N is three}, where
a symmetry argument can be exploited).
\end{rem}

\end{document}